\documentclass[final,onefignum,onetabnum]{siamart171218}
\usepackage{lipsum}
\usepackage{amsfonts, scalefnt,mathtools}
\usepackage{graphicx}
\usepackage{epstopdf}
\usepackage{algorithm}
\usepackage{algorithmic}
\usepackage{braket}
\usepackage{enumitem}
\usepackage{psfrag}
\usepackage{amssymb,bbm,latexsym}
\usepackage{colortbl,fancyhdr}
\usepackage{caption}
\usepackage{float,booktabs,rotating}
\usepackage{amsopn}
\usepackage[percent]{overpic}
\usepackage{tikz,xcolor}
\usepackage{accents}
\usepackage{longtable}

\newcommand{\mc}{\mathcal}

\newcommand{\m}[1]{{\bf{#1}}}
\newcommand{\tr}{^{\sf T}}
\newcommand{\g}[1]{\mbox{\boldmath $#1$}}

\newtheorem{rem}{{\sc Remark}}
\crefname{rem}{Remark}{Remarks}
\numberwithin{rem}{section}
\newtheorem{cor}{{\sc Corollary}}
\crefname{cor}{Corollary}{Corollaries}
\numberwithin{cor}{section}
\newtheorem{lem}{{\sc Lemma}}
\crefname{lem}{Lemma}{Lemmas}
\numberwithin{lem}{section}
\newtheorem{assumption}{{\sc Assumption}}
\crefname{assumption}{Assumption}{Assumptions}

\numberwithin{rem}{section}

\newcommand{\GP}{{GP}}
\newcommand{\NLCO}{{NLCO}}
\newcommand{\NLCOCG}{{NLCO-CG}}
\newcommand{\NPASA}{{NPASA}}
\newcommand{\NPASACG}{{NPASA-CG}}
\newcommand{\PASA}{{PASA}}

\newcommand{\IPOPT}{{IPOPT}}

\newcommand{\PRIMME}{{PRIMME}}

\newcommand{\TheTitle}{A Newton-type Active Set Method for Nonlinear Optimization with Polyhedral Constraints}
\newcommand{\TheTitleShort}{A Newton Active Set Method}
\newcommand{\TheAuthors}{Hager and Tarzanagh}

\headers{\TheTitleShort}{\TheAuthors}

\title{{\TheTitle}}

\author{
WILLIAM W. HAGER \thanks{{\tt hager@ufl.edu},
        http://people.clas.ufl.edu/hager/,
        PO Box 118105,
        Department of Mathematics,
        University of Florida, Gainesville, FL 32611-8105.
        Phone (352) 294-2308. Fax (352) 392-8357.}
\and
DAVOUD ATAEE TARZANAGH  \thanks{{\tt tarzanagh@ufl.edu},
        http://people.clas.ufl.edu/tarzanagh/,
        PO Box 118105,
        Department of Mathematics,
        University of Florida, Gainesville, FL 32611-8105.
}
}

\ifpdf
\hypersetup{
  pdftitle={\TheTitle},
  pdfauthor={\TheAuthors}
}
\fi

\begin{document}

\maketitle

\begin{abstract}
A Newton-type active set algorithm for large-scale minimization subject to polyhedral constraints is proposed. The algorithm consists of a gradient projection step, a second-order Newton-type step in the null space of the constraint matrix, and a set of rules for branching between the two steps. We show that the proposed method asymptotically takes the Newton step when the active constraints are linearly independent and a strong second-order sufficient optimality condition holds. We also show that the method has a quadratic rate of convergence under standard conditions. Numerical experiments are presented illustrating the performance of the algorithm on the CUTEst and on a specific class of problems for which finding second order stationary points is critical.
\end{abstract}



\section{Introduction}
We consider the polyhedral constrained optimization problem
\begin{equation} \label{P}
\min \; \{ f (\m{x}) : \m{x} \in \Omega \}, \quad
\mbox{where} \quad \Omega = \{ \m{x} \in \mathbb{R}^n: \m{Ax} \le \m{b} \}.
\end{equation}
Here, $f$ is a real-valued, twice continuously differentiable function, $\m{A} \in \mathbb{R}^{m \times n}$, $\m{b} \in \mathbb{R}^m$, and $\Omega$ is assumed to be nonempty. We are interested in the case when $n$ and possibly $m$ are large and when second derivatives of $f$ are available.

Optimization problems of the form \eqref{P} arise in many different fields such as machine learning, optimal control, and finance \cite{boyd2011distributed,nocedal2006numerical,stellato2020osqp}. Applications in machine learning include support vector machines \cite{cortes1995support}, constrained neural networks training \cite{chorowski2014learning}, generalized tensor factorization \cite{hong2020generalized,tarzanagh2019regularized}, and estimation of graphical models \cite{banerjee2008model,tarzanagh2018estimation}. In the field of optimal control, moving horizon estimation \cite{allgower1999nonlinear} and model predictive control \cite{garcia1989model} approaches require the solution of a polyhedral constrained problem at each timestep. Financial applications of \eqref{P} include portfolio optimization \cite{cornuejols2006optimization,marko52}. 

Active set methods for polyhedral constrained optimization has been studied quite extensively over the years \cite{fletcher2002nonlinear,nocedal2006numerical}. It relies on a profound theoretical foundation and provides powerful algorithmic tools for the solution of large-scale technologically relevant problems. In recent years, active set Newton methods have proven to be quite effective for solving problems with thousands of variables and constraints but are likely to become very expensive as the problems they are asked to solve become larger \cite{nocedal2006numerical,byrd2003algorithm}. Indeed, the principal cost of these methods is often dominated by that of solving the quadratic programming subproblems, and in the absence of a truncating scheme, this cost can become prohibitive. These concerns have motivated us to look for a different active-set approach.


In this paper, we propose a Newton-type polyhedral active set algorithm (\NPASA{}) which consists of a gradient projection (\GP{}) phase and a second order Newton-type linearly constrained optimizer (\NLCO{}) phase that solves a sequence of equality-constrained quadratic programming over the faces of the polyhedron $\Omega$. The algorithm does not require the solution of a general quadratic program at each iteration. We show that \NPASA{} asymptotically executes \NLCO{} when the active constraints are linearly independent and a strong second-order sufficient optimality condition holds. We also show that \NPASA{} has a quadratic rate of convergence under standard assumptions.

{\bf Related work.}
This work is related to a broad range of literature on constrained optimization.
Some of the initial work focused on
the use of the conjugate gradient method with bound constraints as in \cite{polyak1969conjugate, more1991solution,dostal1997box, dostal2003proportioning,dostal2003augmented, dostal2005minimizing,yang1991class}. Work on gradient projection methods include \cite{yang1991class, bertsekas1974goldstein, calamai1987projected,goldstein1964convex, mccormick1972gradient, schwartz1997family}. Convergence is accelerated using Newton and trust region methods \cite{conn2000trust}. Superlinear and quadratic convergence for nondegenerate problems can be found in \cite{bertsekas1982projected,burke1990convergence, conn1988global, facchinei1998active,forsgren1997newton}, while analogous convergence results are given in \cite{facchinei2002truncated, friedlander1994new, lescrenier1991convergence, lin1999newton}, even for degenerate problems.

This work is also related to a substantial body of literature on active set methods for quadratic programming problem \cite{forsgren2015active,gill2015methods,izmailov2014newton,friedlander2008global}. A treatment of active set methods in a rather general setting is given in \cite{izmailov2014newton}. Most directly related to our work is \cite{goldberg2015active} where a very efficient two-phase active set method for conic-constrained quadratic programming is proposed. As in \cite{hager2006new,hager2016active,friedlander2008global,zhang2020smoothing} the first phase of the conic active-set method \cite{goldberg2015active} is the GP method, while the second phase is the Newton method in the reduced subspace defined by the free variables and the boundary of the active cones. Although this method uses the second-order information, the convergence to second-order stationary points was not established. Further, it remains unclear whether this method asymptotically takes the Newton step or not.

\textbf{Outline.} Section~\ref{sec:2} gives a detailed statement of the proposed active set algorithm, while Section~\ref{sec:3} establishes its global convergence. Section~\ref{sec:4} shows that asymptotically \NPASA{} performs only the Newton step when the active constraint gradients are linearly independent and a strong second-order sufficient optimality condition holds. Section~\ref{sec:results} provides some numerical results. Finally, Section~\ref{sec:conclusions} concludes the paper.

{\bf Notations.}
For any set $\mc{S}$, $|\mc{S}|$ stands for the number of elements (cardinality) of $\mc{S}$, while $\mc{S}^c$ is the complement of $\mc{S}$.  The subscript $k$ is often used to denote the iteration number in an algorithm, while $\m{x}_{ki}$ stands for the $i$-th component of the iterate $\m{x}_k$. The function value, gradient and Hessian at point $\m{x}_k$ is denoted by $f_k = f(\m{x}_k)$, $\m{g}_k = \m{g}(\m{x}_k)$, and $\m{H}_k = \m{H}(\m{x}_k)$, respectively. For a symmetric matrix $\m{M}$, we use the notation $\sigma_{\min} (\m{M}) \geq 0$ for $\m{M}$ positive semidefinite and $\sigma_{\min} (\m{M} ) > 0$ for $\m{M}$ positive definite. $\m{I}_n$ denotes the identity matrix of size $n$. The ball with center $\m{x}$ and radius $r$ is denoted $\mc{B}_{r}(\m{x})$. For any matrix $\m{M}$, $\mc{N}(\m{M})$ is the null space. If $\mc{S}$ is a subset of the row indices of $\m{M}$,
then ${M}_{\mc{S}}$ denotes the submatrix of $\m{M}$ with row indices $\mc{S}$.
For any vector $\m{b}$, $\m{b}_{\mc{S}}$
is the subvector of $\m{b}$ with indices  $\mc{S}$.
$\mc{P}_\Omega (\m{x})$ denotes the Euclidean projection of $\m{x}$ onto $\Omega$:
\begin{equation} \label{proj}
\mc{P}_{\Omega} (\m{x}) =
\arg  \; \min \{ \|\m{x} - \m{y}\| : \m{y} \in \Omega \} .
\end{equation}
For any $\m{x} \in \Omega$, the active and free index sets are
defined by
\[
\mc{A}(\m{x}) = \{ i: (\m{Ax} - \m{b})_i = 0\} \quad \mbox{and} \quad
\mc{F}(\m{x}) = \{ i: (\m{Ax} - \m{b})_i < 0\},
\]
respectively.

\section{Algorithm structure}\label{sec:2}

In the proposed active set method, it takes either the iteration of the \GP{} algorithm \cite{calamai1987projected} or the iteration of the \NLCO{} based on some switching rules. Algorithm~\ref{alg:pga} is a simple monotone \GP{} algorithm using an Armijo-type linesearch \cite{armijo1966minimization}. Better numerical performance is achieved with a more general nonmonotone linesearch such as that given in \cite{birgin2000nonmonotone}, and all the analysis directly extends to this more general framework.

\begin{algorithm}[t]
\caption{ Gradient Projection (\GP{})}
\label{alg:pga}
    \begin{algorithmic}
            \STATE \textbf{Parameters}: $\delta_1$ and $\eta \in (0, 1)$.
\FOR{$k=1, 2 , \ldots $}
\STATE Set $\m{d}_k= -  \nabla f(\m{x}_k)$ and update
\begin{eqnarray}\label{xx-def}
\m{x}_{k+1} &=& \mc{P}_\Omega (\m{x}_k +s_k \m{d}_k),
\end{eqnarray}				
where $s_k=\max\{\eta^0, \eta^1, \ldots \}$ is chosen such that
\begin{equation}\label{SuffDecrCond1}
f(\m{x}_{k+1}) \leq f(\m{x}_{k})+\delta_1 \langle \nabla f(\m{x}_{k}), \m{x}_{k+1}-\m{x}_k \rangle ;
\end{equation}			
 \ENDFOR
    \end{algorithmic}
\end{algorithm}

Let $\m{x}_k \in \Omega $ be the current iterate and assume the \NLCO{} be chosen to get $\m{x}_{k+1}$. Then, the \NLCO{} solves the linearly constrained problem
\begin{equation}\label{P-face}
\min_{} \; \{ f (\m{y}) :   \m{y} \in \widehat{\Omega}(\m{x}_k)\}, 
\end{equation}
where 
\begin{align}\label{eqn:face:set}
\widehat{\Omega}(\m{x}_k)= \big\{ \m{y}:  (\m{Ay} - \m{b})_i = 0
\mbox{ for all } i \in \mc{A}(\m{x}_k) \big\}.
\end{align}
%

A common approach for solving \eqref{P-face} is to eliminate the constraints and solve a \textit{reduced} problem~\cite{fletcher2002nonlinear,nocedal2006numerical}. More specifically, let  $\m{Z}_{\m{A}_{\mc{I}}}$ be a matrix whose columns are mutually orthogonal and span the null space $\mc{N}(\m{A}_\mc{I})$. Then, $\m{Z}_{\m{A}_{\mc{I}}}$ has dimension $n \times (n-|\mc{I}|)$ and is such that 
\begin{equation}\label{eqn:zprop}
\m{Z}_{\m{A}_{\mc{I}}}^\top \m{Z}_{\m{A}_{\mc{I}}}=\m{I}_{n-|\mc{I}|}, \qquad   \m{A}_{\mc{I}} \m{Z}_{\m{A}_{\mc{I}}} =\m{0}. 
\end{equation}

Under this notation and for $\mc{I}=\mc{A}(\m{x}_k)$, the problem defined at \eqref{P-face} can be redefined as the reduced problem
\begin{equation} \label{P-uncon}
\min \; \{ f (\m{x}_k+ \m{Z}_{\m{A}_{\mc{I}}} \m{p}) :  \m{p} \in  \mathbb{R}^{n-|\mc{I}|}\}. 
\end{equation}
This is now an unconstrained problem in the $(n-|\mc{I}|) \times 1 $ vector $\m{p}$ and the methods for the unconstrained problem can be applied. When evaluating first and second derivatives with respect to the variable $\m{p}$ we see that the gradient vector is given by $ \m{Z}_{\m{A}_{\mc{I}}}^\top \m{g}(\m{x}_k)$. This vector will be defined as the reduced gradient and denoted by $ \m{g}^{\mc{I}}(\m{x}_k)$. We also let  $\m{g}^{\mc{A}}(\m{x}_k)$ denote $ \m{g}^{\mc{I}}(\m{x}_k)$ for $\mc{I} = \mc{A}(\m{x}_k)$. Similarly, the $(n-|\mc{I}|)\times(n-|\mc{I}|)$ Hessian matrix with respect to $\m{p}$ is given by $\m{Z}_{\m{A}_{\mc{I}}}^\top \m{H}(\m{x}_k) \m{Z}_{\m{A}_{\mc{I}}}$. This will be defined as the reduced Hessian matrix at $\m{x}_k$ and written as $\m{H}^{\mc{I}} (\m{x}_k)$. We also use $\m{H}^{\mc{A}}(\m{x}_k)$ to denote  $\m{H}^{\mc{I}} (\m{x}_k)$ for  $\mc{I} = \mc{A}(\m{x}_k)$.  Now, it is possible to construct a local quadratic model for \eqref{P-uncon} from the corresponding Taylor series expansion of $f$ at $\m{x}_k$ as 
\begin{equation} \label{P-quad}
 \min \; \{ Q(\m{p}):  \m{p} \in  \mathbb{R}^{n-|\mc{I}|}\},
\end{equation}
where 
\begin{equation*}
Q(\m{p}):= \m{p}^{\top}  \m{g}^{\mc{I}}(\m{x}_k) +  \m{p}^{\top} \m{H}^{\mc{I}} (\m{x}_k)\m{p}.
\end{equation*}

Having characterized the reduced problem~\eqref{P-uncon} and its local quadratic model \eqref{P-quad}, we now describe the proposed \NLCO{} algorithm. We use a standard line-search framework \cite{nocedal2006numerical}.  Given the current iterate $\m{x}_k \in \Omega$, the new iterate is obtained as
\begin{equation}\label{eqn:iterate}
\m{x}_{k+1} = \m{x}_k + s_k  \m{d}_k,
\end{equation}
where $\m{d}_k= \m{Z}_{\m{A}_{\mc{I}}} \m{p}_k$  is a chosen search direction and $s_k>0$ is a stepsize computed by a backtracking linesearch procedure.

In Algorithm~\ref{algo:nlco}, we provide a general framework for solving problem \eqref{P-uncon}. This algorithm alternates between (perturbed) negative curvature and Newton-like steps using the minimum eigenvalue of the reduced Hessian matrix. (Either can be taken first; arbitrarily, we state our algorithm and analysis assuming that one starts with a perturbed negative curvature step.) At a given iterate $\m{x}_k \in \Omega$, let $\sigma_k$ denote the left-most eigenvalue of the reduced Hessian matrix $\m{H}^{\mc{A}} (\m{x}_k)$. If $\sigma_k < -\epsilon_H $ for some $ \epsilon_H \in [0, \infty)$, a curvature direction $\m{u}_k$ is computed such that
\begin{align}\label{eq:negcurv}
    \m{u}_k^\top \m{H}^{\mc{A}}(\m{x}_k)  \m{u}_k  < 0,  \qquad  
   \m{u}_k^\top \m{g}^{\mc{A}} (\m{x}_k)  \leq 0,  \qquad \|\m{u}_k\| = |\sigma_k|.
\end{align}
Then, the algorithm sets $\m{p}_k = \m{p}^{c}_k $, where $\m{p}^{c}_k $ is a combination of the negative gradient direction with the negative curvature direction; that is, 
\begin{align}\label{eq:neggrad}
\m{p}^{c}_k = \m{u}_k -\alpha \m{g}^{\mc{A}}(\m{x}_k)
\end{align} 
 for some $\alpha \in (0,\infty)$.  
 
If $|\sigma_k| \leq \epsilon_H$, the algorithm sets $\m{p}_k = \m{p}^{r}_k$,  where $\m{p}^{r}_k$ is obtained by solving the regularized Newton system
\begin{equation}\label{eq:regnewton}
 \big(\m{H}^{\mc{A}}(\m{x}_k)+(|\sigma_k|+ \epsilon_R) \m{I}\big) \m{p}^{r}_k = -\m{g}^{\mc{A}}(\m{x}_k)
\end{equation}
for some $\epsilon_R \in (0, \infty)$.  

If $\m{p}_k$ has not yet been chosen, the algorithm sets $\m{p}_k = \m{p}^{n}_k$, where $\m{p}^{n}_k$ is an exact minimizer of the subproblem \eqref{P-quad} and can be obtained by solving the Newton system
\begin{equation}\label{eq:newton}
\m{H}^{\mc{A}}(\m{x}_k) \m{p}^{n}_k = -\m{g}^{\mc{A}}(\m{x}_k).
\end{equation}

\begin{algorithm}[t]
\caption{Newton-type Linearly Constrained Optimizer (\NLCO)}\label{algo:nlco}
\begin{algorithmic}
\STATE \textbf{Parameters}: $  \epsilon_H \in [0, \infty);  \epsilon_R, \alpha \in (0, \infty)$; $ \delta_1 \in (0, 1)$; and $\delta_2 \in [\delta_1, 1)$.
\FOR{$k=1, 2 , \ldots $}
\STATE  $\m{x}_k \in \Omega$ and $\mc{I} \leftarrow \mc{A}(\m{x}_k)$;
\STATE Evaluate the reduced gradient $\m{g}^{\mc{I}}(\m{x}_k) $ and the reduced Hessian $\m{H}^{\mc{I}} (\m{x}_k)$;
\STATE Compute $\sigma_k$ as the minimum eigenvalue of $\m{H}^{\mc{I}}(\m{x}_k)$;
\IF{$\sigma_{k} < - \epsilon_H$} 
\STATE Set $\m{p}_k = \m{p}^{c}_k$, where $\m{p}^c_k$ is defined by \eqref{eq:neggrad}; 
\ELSIF{$|\sigma_k| \leq \epsilon_H$} 
\STATE Set $\m{p}_k = \m{p}^{r}_k$, where $\m{p}^{r}_k$ is obtained by solving \eqref{eq:regnewton}; 
\ELSE
\STATE Set $\m{p}_k = \m{p}^{n}_k$, where $\m{p}^{n}_k$ is obtained by solving \eqref{eq:newton};
\ENDIF
\STATE Set $\m{d}_k= \m{Z}_{\m{A}_{\mc{I}}} \m{p}_k$;
\STATE Find the stepsize $s_k$ such that \eqref{eq:line1} and \eqref{eq:line2} hold;
\STATE Update $\m{x}_{k+1} =\m{x}_k+ s_k \m{d}_k$;
\ENDFOR
\end{algorithmic}
\end{algorithm}

Once a search direction has been selected, a backtracking linesearch is applied to determine $s_k$. We follow \cite{mccormick1977modification,more1979use,goldfarb1980curvilinear} in the linesearch and adapt it to cope with polyhedral constraints. For the sake of completion, the linesearch is reviewed here, and the properties that are subsequently required for the linearly constrained case are given in Lemmas \ref{lem:sk} and \ref{thm:glob:face}. 

Let  
\begin{equation}\label{eq:linf}
 \phi_k(s) := f(\m{x}_k+ s \m{d}_k) \quad \text{and} \quad  \psi_k(s):= \phi_k^{'}(0)+ \frac{1}{2} \min \{ \phi_k^{''}(0) , 0\} s.
\end{equation}
The stepsize $s_k$ in \eqref{eqn:iterate} must provide a sufficient reduction in the objective function according to
\begin{subequations}
\begin{align}\label{eq:line1}
 \phi_k(s_k) &\leq  \phi_k(0)+ \delta_1 \psi_k(s_k)s_k 
\end{align}
and a sufficient reduction in curvature by
\begin{align}\label{eq:line2}
| \phi_k^{'}(s_k)| &\leq  \delta_2 |\psi_k(s_k)|,
\end{align}
\end{subequations}
where $\delta_1 \in (0, 1)$ and $\delta_2 \in [\delta_1, 1)$.  The linesearch is designed to give $ \lim_{k \to \infty}\phi_k^{'}(0)=0 $ and $\liminf_{k\to\infty} \phi_k^{''}(0)\geq0$. Mor\'{e} and Sorensen \cite[Lemma 5.2]{more1979use} provided a proof for the existence of a stepsize satisfying \eqref{eq:line1} and \eqref{eq:line2}. For completeness, the lemma is reproduced here and the proof is provided in Appendix.
\begin{lemma}\label{lem:sk}
Let $\phi : \mathbb{R}\rightarrow \mathbb{R}$ be twice continuously differentiable in an open interval $S$ that contains the origin, and suppose that $\{s \in S : \phi(s) \leq \phi(0)\}$ is compact. Let $\delta_1 \in (0, 1)$ and $\delta_2 \in [\delta_1, 1)$. If $\phi^{'}(0) <0 $, or if  $\phi^{'}(0) \leq 0$ and $\phi^{''}(0) < 0$, then there is an $s >0$ in $S$ such that \eqref{eq:line1} and \eqref{eq:line2} hold.
\end{lemma}

The implementation of \NLCO{} involves several decisions concerning practical details; we describe some of them in the following paragraphs:
\begin{itemize}
\item[--] Negative curvature direction $\m{u}_k$ can be obtained at a relatively small cost. It can be obtained by executing the power iteration method on the reduced Hessian matrix $\m{H}^{\mc{A}}(\m{x}_k)$ to obtain an eigenvector corresponding to the leftmost eigenvalue $\sigma_k$ so that \eqref{eq:negcurv} holds. More efficient and robust methods using variants of Lanzos algorithm to compute the algebraically smallest eigenpair can be found in \cite{stathopoulos2010primme}. It is worth mentioning that at most one eigenvector computation and one linear system solving are needed per iteration of Algorithm~\ref{algo:nlco}, along with a reduced gradient evaluation and the reduced Hessian-vector multiplication. 
\item[--] In Algorithm~\ref{algo:nlco}, we require that
\begin{equation}\label{eqn:act:seq}
 \mc{A}(\m{x}_{k}) \subseteq \mc{A}(\m{x}_{k+1}) \qquad \textnormal{for all} \quad k =1, 2, \ldots~.
\end{equation}
This condition is easily fulfilled by our algorithm which adds constraints to the active set whenever a new constraint becomes active.
\end{itemize}

\begin{algorithm}[t]
\caption{Newton-type Polyhedral Active Set Algorithm (\NPASA{})}\label{algo:npasa}
\begin{algorithmic}
\STATE \textbf{Parameters}: $\epsilon_E,  \epsilon_H \in [0, \infty);  \epsilon_R, \alpha \in (0, \infty)$; $\eta,\theta, \mu, \delta_1 \in (0, 1)$; and $\delta_2 \in [\delta_1, 1)$.
\STATE Choose $\m{x}_0 \in \mathbb{R}^n$.
\STATE Set $k=0$ and  $\m{x}_{k+1} = \mc{P}_\Omega (\m{x}_k)$.
\STATE \textbf{Phase~1 (First Order)}:
\WHILE{ $E(\m{x}_{k}) > \epsilon_E$ }
\STATE Execute the GP step to obtain $\m{x}_{k+1}$ from $\m{x}_k$;
\STATE  Set $k \leftarrow k+1$;
\IF{$e(\m{x}_{k}) \leq  \theta E(\m{x}_{k})$} 
\STATE Set $\theta \leftarrow \mu \theta$;
\ENDIF
\IF{$e(\m{x}_{k}) > \theta E(\m{x}_{k})$}
\STATE Go to Phase~2;
\ENDIF
\ENDWHILE
\STATE \textbf{Phase~2 (Second Order)}:
\WHILE{ $E(\m{x}_{k}) > \epsilon_E$ }
\STATE Execute the NLCO step to obtain $\m{x}_{k+1}$ from $\m{x}_k$; 
\STATE Set $k \leftarrow k+1 $;
\IF{$e(\m{x}_{k}) \leq \theta E(\m{x}_{k})$} 
\STATE Set $\theta \leftarrow \mu \theta$;
\STATE Go to Phase~1;
\ENDIF
\ENDWHILE
\end{algorithmic}
\end{algorithm}

Following \cite{hager2006new,hager2016active,zhang2020smoothing}, we next provide some rules for branching between \GP{} and \NLCO{} steps in our proposed active set algorithm, \NPASA{}. To do so, we define $
 e(\m{x}) = \|\m{g}^\mc{A} (\m{x})\|$ as a local measure of stationarity in the sense that it vanishes if and only if $\m{x}$ is a stationary point on its associated face $$\widehat{\Omega}(\m{x}) =\big\{ \m{y}:  (\m{Ay} - \m{b})_i = 0
\mbox{ for all } i \in \mc{A}(\m{x}) \big\}.$$

Let $\nabla_{\Omega} f(\m{x})$ denote the projected gradient of $f$ at a point $\m{x} \in \Omega$ \cite{calamai1987projected}:
\begin{equation} \label{pg_grad_minprob}
\nabla_{\Omega} f(\m{x}) :=\mc{P}_{\mc{T}(\m{x})} [- \m{g}(\m{x})]= \arg\min \left\{ \Vert \m{d} + \m{g}(\m{x})\Vert
\quad \mbox{ s.t.} \quad \m{d} \in \mc{T}(\m{x}) \right\}.
\end{equation}
Here, $\mc{T}(\m{x})$ is the tangent cone to $\Omega$ at $\m{x}$, i.e.,   
\begin{equation}\label{eqn:swi5}
\mc{T}(\m{x}) = \Big\{ \m{d}~:~\m{a}_i^\top \m{d} \leq 0 \mbox{ for all } i  \in \mc{A}(\m{x})\Big\}.
\end{equation}
The projected gradient $\nabla_{\Omega} f(\m{x})$ can be used to characterize stationary points because if $\Omega$ is a convex set, then $\m{x} \in \Omega$ is a stationary point of problem \eqref{P} if and only if $\nabla_{\Omega} f(\m{x})=0$. We monitor the convergence to a stationary point using the function $E$ defined by
\[
E(\m{x}) = \|\nabla_{\Omega} f(\m{x}) \|.
\]

As shown in \cite{calamai1987projected}, the limit points of a bounded sequence $\{ \m{x}_k \}$ generated by any \GP{} algorithm are stationary, and
\begin{equation*} 
\lim_{k \rightarrow \infty} E(\m{x}_k)= 0,
\end{equation*}
provided the steplengths are bounded and satisfy suitable sufficient decrease conditions. Further, if an algorithm is able to drive the projected gradient toward
zero, then it is able to identify the active variables that are nondegenerate at the solution within finite number of iterations. This is essential in providing the quadratic convergence result of our proposed active set algorithm.

The rules for switching between phase one and phase two depend on the relative size of the stationarity measures $E$ and $e$. We choose a parameter $\theta \in (0, 1)$ and branch from phase one to phase two when $e(\m{x}) \ge \theta E(\m{x})$. Similarly, we branch from phase two to phase one when $e(\m{x}) < \theta E(\m{x})$. To ensure that only phase two is executed asymptotically at a degenerate stationary point, we may need to decrease $\theta$ as the iterates converge.

Algorithm~\ref{algo:npasa} is the Newton-type polyhedral active set algorithm.
The parameter $\epsilon_E$ is the convergence tolerance, $\theta$ controls the branching between phase one and phase two, while $\mu$ controls the decay of $\theta$.

\section{Global convergence} \label{sec:3}

In this section, we establish the convergence properties of Algorithm~\ref{algo:npasa}. To do so, we make the following assumptions throughout.

\begin{assumption}\label{assu-a} 
The level set $\mc{L}_f(\m{x}_1) := \{ \m{x} \in \Omega : f (\m{x}) \le f(\m{x}_1) \} $ is compact.
\end{assumption}
\begin{assumption}\label{assu-b}
The function $f$ is twice continuously differentiable on an open neighborhood of $\mc{L}_f(\m{x}_1)$ that includes the trial points generated by the
algorithm. Further, $\nabla^2 f(\m{x})$ is bounded on this neighborhood.
\end{assumption}
\begin{rem}\label{rem:bound}
Under Assumptions~\ref{assu-a} and \ref{assu-b}, there exist finite scalars $f_{\textnormal{low}}$, $U_g > 0$, and $U_H > 0$ such that,  $f(\m{x}) \geq f_{\textnormal{low}}$, $\|\m{g}(\m{x})\| \leq U_g$,  and $\|\m{H}(\m{x})\| \leq U_H$ for all $\m{x} \in
\mc{L}_f(\m{x}_1)$. We observe that $U_H$ is a Lipschitz constant for the gradient $\m{g}(\m{x})$.
\end{rem}

The following technical lemma derives bounds on the norm of the search direction obtained from each type of step and shows that the search directions yield descent.
\begin{lemma}\label{lem:bound_dk}
Under Assumptions~\ref{assu-a} and \ref{assu-b}, there exist positive constants $c_1$ and $c_2$ such that, for all $k$,
\begin{align*}
\|\m{d}_k\| \leq c_1  \quad \textnormal{and} \quad   {\m{g}(\m{x}_k)}^\top  \m{d}_k \leq - c_2  \| \m{g}^{\mc{A}}(\m{x}_k)\|^2.
\end{align*}

\end{lemma}
\begin{proof}
%
Let $\mc{I} =\mc{A}(\m{x}_k)$. We consider three disjoint cases:
\\
\textbf{Case 1}: $\m{d}_k = \m{Z}_{\m{A}_{\mc{I}}}\m{p}^c_k$. Then, it follows from \eqref{eq:neggrad} that
\begin{align}\label{eqn:up:negcur}
\nonumber
\|\m{d}_k\| &= \|\m{Z}_{\m{A}_{\mc{I}}}(\m{u}_k-\alpha  \m{Z}_{\m{A}_{\mc{I}}}^\top \m{g} (\m{x}_k))\| \\
\nonumber
 &\leq \|\m{Z}_{\m{A}_{\mc{I}}}\m{u}_k \| + \alpha \| \m{Z}_{\m{A}_{\mc{I}}} \m{Z}_{\m{A}_{\mc{I}}}^\top \m{g} (\m{x}_k)\| \\
 \nonumber
 & \leq | \sigma_k| + \alpha \| \m{g}(\m{x}_k)\| 
\\& \le  U_H + \alpha U_g,
\end{align}
where the second inequality is obtained from \eqref{eq:negcurv} and \eqref{eqn:zprop}; and the last inequality uses Remark~\ref{rem:bound}.  We also note that  
\begin{align}\label{eq:negdes}
\nonumber
{\m{g}(\m{x}_k)}^\top  \m{d}_k &= {\m{g}(\m{x}_k)}^\top  \m{Z}_{\m{A}_{\mc{I}}} ( \m{u}_k-\alpha \m{g}^\mc{I}(\m{x}_k)) \\
\nonumber
&  = \m{u}_k^\top \m{g}^\mc{I}(\m{x}_k)- \alpha {\m{g}(\m{x}_k)}^\top   \m{Z}_{\m{A}_{\mc{I}}} \m{g}^\mc{I}(\m{x}_k) \\
& \leq  -\alpha \|\m{g}^\mc{I}(\m{x}_k)\|^2,
\end{align} 
where the inequality follows from \eqref{eq:negcurv}.
\\
\textbf{Case 2}: $\m{d}_k =  \m{Z}_{\m{A}_{\mc{I}}}\m{p}^r_k$. We can suppose that $\|\m{g}^{\mc{I}}(\m{x}_k)\|>0$ because otherwise $\|\m{d}_k\|=0$. Then, from  \eqref{eq:regnewton}, we have
\begin{eqnarray}\label{eqn:up:rnew}
\nonumber
\|\m{d}_k\| & \le &  \left \| \m{Z}_{\m{A}_{\mc{I}}} \big(\m{Z}_{\m{A}_{\mc{I}}}^\top  \m{H}(\m{x}_k) \m{Z}_{\m{A}_{\mc{I}}}+ (|\sigma_k|+\epsilon_R) \m{I}_n \big)^{-1} \m{Z}_{\m{A}_{\mc{I}}}^\top \m{g}(\m{x}_k) \right\| \\
\nonumber
 &\le & \left\|\left(\m{Z}_{\m{A}_{\mc{I}}}^\top  \m{H}(\m{x}_k) \m{Z}_{\m{A}_{\mc{I}}} +(|\sigma_k|+\epsilon_R) \m{I}_n \right)^{-1} \right\| \ \|\m{g}^{\mc{I}}(\m{x}_k)\| \\ 
&\le & \epsilon_R^{-1} \|\m{g}^{\mc{I}}(\m{x}_k)\| \le  \epsilon_R^{-1}  U_g ,
\end{eqnarray}
where the second inequality uses \eqref{eqn:zprop} and the last inequality follows from Remark~\ref{rem:bound}. Further, we observe that
\begin{align}\label{eq:newdec}
\nonumber
{\m{g}(\m{x}_k)}^\top  \m{d}_k  &= - {\m{g}(\m{x}_k)}^\top  \m{Z}_{\m{A}_{\mc{I}}} \left(\m{Z}_{\m{A}_{\mc{I}}}^\top  \m{H}(\m{x}_k) \m{Z}_{\m{A}_{\mc{I}}}+ (|\sigma_k|+\epsilon_R) \m{I}_n  \right)^{-1} \m{Z}_{\m{A}_{\mc{I}}}^\top \m{g}(\m{x}_k)  \\
&\leq -  \left\|\m{Z}_{\m{A}_{\mc{I}}}^\top  \m{H}(\m{x}_k) \m{Z}_{\m{A}_{\mc{I}}} +(|\sigma_k|+\epsilon_R) \m{I}_n \right\| ^{-1} \|{\m{g}^{\mc{I}}(\m{x}_k)}\|^2 \\
&\leq - \left(2 U_H+\epsilon_R \right)^{-1} \|{\m{g}^{\mc{I}}(\m{x}_k)}\|^2,
\end{align}
where the last inequality uses  Remark~\ref{rem:bound}. \\
\textbf{Case 3}: $\m{d}_k =\m{Z}_{\m{A}_{\mc{I}}} \m{p}^n_k$. In this case, the reduced Hessian is positive definite with smallest eigenvalue bounded below by $\epsilon_H$. Similar to the previous case, suppose that $\|\m{g}^{\mc{I}}(\m{x}_k)\|>0$. Then, from \eqref{eq:newton}, we obtain
\begin{eqnarray}\label{eqn:up:new}
\nonumber
\|\m{d}_k\|  & \le &    \| \m{Z}_{\m{A}_{\mc{I}}} \big(\m{Z}_{\m{A}_{\mc{I}}}^\top  \m{H}(\m{x}_k) \m{Z}_{\m{A}_{\mc{I}}}\big)^{-1} \m{Z}_{\m{A}_{\mc{I}}}^\top \m{g}(\m{x}_k) \| \\
\nonumber
 &\le& \sigma_k^{-1} \|\m{g}^{\mc{I}}(\m{x}_k)\|\\
                &\le&  \epsilon_H^{-1} \|\m{g}^{\mc{I}}(\m{x}_k)\| \le  \epsilon_H^{-1} U_g ,
\end{eqnarray}
where the third inequality follows since  $\sigma_k >\epsilon_H$ and the last inequality uses Remark~\ref{rem:bound}. Further, similar to the previous case, we have
\begin{align}\label{eq:newdecc}
{\m{g}(\m{x}_k)}^\top  \m{d}_k \leq - U_H^{-1}  \| \m{g}^{\mc{I}}(\m{x}_k)\|^2.
\end{align}
Now, it follows form  \eqref{eqn:up:negcur}, \eqref{eqn:up:rnew} and \eqref{eqn:up:new} that
\begin{equation*}
c_1= \max \left\{U_H + \alpha U_g,\epsilon_R^{-1}  U_g ,\epsilon_H^{-1} U_g \right\}. 
\end{equation*} 
Similarly, using \eqref{eq:negdes}, \eqref{eq:newdec} and \eqref{eq:newdecc}, we obtain
\begin{equation*}
c_2= \max \left\{\alpha, U_H^{-1} \right\}. 
\end{equation*} 
\end{proof}

The following Lemma~\ref{thm:glob:face} gives some properties of the iterates for a sequence generated by the linesearch conditions \eqref{eq:line1} and \eqref{eq:line2}. In particular, it shows that \NLCO{} converges toward second-order stationary.

\begin{lemma}\label{thm:glob:face}
Given Assumptions~\ref{assu-a} and \ref{assu-b}, assume that a sequence $\{\m{x}_k\}_{k=1}^{\infty}$ is generated by \NLCO{} with $\epsilon_H=0$. Then,
\begin{enumerate}[label=(\roman*)]
\item  \label{itm:fir} $ \lim_{k\to\infty} \|\m{g}^{\mc{A}}(\m{x}_k)\|=0$; and
\item \label{itm:sec} $ \liminf_{k\to\infty} \sigma_{\min}(\m{H}^{\mc{A}} (\m{x}_k))\geq 0$.
\end{enumerate}
\end{lemma}
\begin{proof}
It follows from \eqref{eq:line1} that
$$
f(\m{x}_k)- f(\m{x}_k+s_k \m{d}_k) \geq -\delta_1(\phi_k^{'}(0)s_k+ \frac{1}{2} \min \{ \phi_k^{''}(0) , 0\} s_k^2).
$$
Since $\delta_1>0$ and by Remark~\ref{rem:bound},  $f (\m{x}_k)$ is bounded below for all $k$, we obtain
\begin{align}\label{eqn:conv:phi}
\lim_{ k \rightarrow \infty} \phi_k^{'}(0)s_k =0
, \quad \text{and} \quad \lim_{ k \rightarrow \infty} \min\{\phi_k^{''}(0),0\}s_k^2 =0.
\end{align}

 \ref{itm:fir}.  We first show that 
\begin{equation}\label{lim:gd}
\lim_{k\to\infty} \phi_k^{'}(0)  =0, \quad \text{where} \quad \phi_k(s) = f(\m{x}_k+ s \m{d}_k).
\end{equation}
Assume by contradiction there is a subsequence $I^{'}$ and a positive constant $\epsilon_1$ such that
\begin{equation}\label{eqn:cont:grad}
 \phi_k^{'}(0)\leq  - \epsilon_1 <0, \quad \text{for} \quad k \in I^{'}.
\end{equation}
It follows from \eqref{eq:line2} that the stepsize $s_k$ satisfies
\begin{align}\label{eq:lem21}
\phi_k^{'}(s_k) \geq \delta_2 \psi_k(s_k),
\end{align}
where $ \psi_k(s)= \phi_k^{'}(0)+ \frac{1}{2} \min \{ \phi_k^{''}(0) , 0\} s$.  Since $\phi_k^{'}(0)$ is a continuously differentiable univariate function, the mean-value theorem ensures the existence of a $ t_k \in (0, s_k)$ such that
\begin{align}\label{eq:lem22}
 \phi^{'}_k(s_k) \leq \phi^{'}_k(0) + s_k \phi^{''}_k(t_k).
\end{align}
Combine \eqref{eq:lem21} with \eqref{eq:lem22} to get 
\begin{eqnarray}\label{eqn:bound:step}
\nonumber
s_k(\phi^{''}_k(t_k)+ \frac{\delta_2}{2} \max \{- \phi^{''}_k(0),0 \}) &\geq& - (1-\delta_2)\phi_k^{'}(0) \\
&\geq & \epsilon_1 (1-\delta_2) \qquad   \text{for} \quad k \in I',
\end{eqnarray}
where the last inequality uses \eqref{eqn:cont:grad}. From \eqref{eqn:bound:step}, we get $\lim \sup_{k\in I^{'}} s_k \neq 0$. Now, from \eqref{eqn:conv:phi}, we have $\lim_{ k \rightarrow \infty} \phi_k^{'}(0) =0$. This is a contradiction to \eqref{eqn:cont:grad}. Hence, the assumed existence of $I^{'}$ is false and \eqref{lim:gd} holds.

We now show that  $\lim_{k\to\infty} \|\m{g}^{\mc{A}}(\m{x}_k)\| = 0$. From Lemma~\ref{lem:bound_dk}, we have 
\begin{align*}
\phi_k^{'}(0) = {\m{g}(\m{x}_k)}^\top  \m{d}_k \leq -c_2\| \m{g}^\mc{A}(\m{x}_k)\|^2
\end{align*}
for some $c_2 >0$.  This, together with \eqref{lim:gd} implies that $ \|\m{g}^\mc{A}(\m{x}_k)\|\to 0$. 
\bigskip

\ref{itm:sec}. We first show that 
\begin{equation}\label{eq:secbound}
\liminf_{k \to \infty} \phi^{''}_k(0) \geq 0.
\end{equation}
Assume by contradiction that there is a subsequence $I^{''} $ such that $\phi^{''}_k(0)  \leq -\epsilon_2 <0 $ for $k \in I^{''}$. By the same argument as in \eqref{eq:lem21}--\eqref{eqn:bound:step},  there exists  a $ t_k \in (0, s_k)$ such that 
\begin{eqnarray*}
\nonumber
s_k(\phi^{''}_k(t_k)+ \delta_2 \max \{- \phi^{''}_k(0),0 \}) &\geq& - (1-\delta_2)\phi_k^{'}(0) \geq 0,
\end{eqnarray*}
where the second inequality follows since $\phi_k^{'}(0) \leq 0$. Thus, 
\begin{equation}\label{lemeq1}
\phi^{''}_k(t_k) \geq \delta_2\phi^{''}_k(0).
\end{equation}
On the other hand, our assumption $\phi^{''}_k(0)  <0$ together  with \eqref{eqn:conv:phi} gives $\lim_{k\in I^{''}} s_k = 0$. Thus, \eqref{lemeq1} cannot hold for $k$ sufficiently large. Consequently, the assumed existence of $I^{''}$ is false, and \eqref{eq:secbound} holds. 

Now, if $\epsilon_H=0$ and \ref{itm:sec} does not hold then by Algorithm~\ref{algo:nlco} and for $k$ sufficiently large, we have $\m{d}_k= \m{Z}_{\m{A}_{\mc{I}}}(\m{u}_k-\alpha  \m{Z}_{\m{A}_{\mc{I}}}^\top \m{g} (\m{x}_k))$. From \ref{itm:fir}, we have $\lim_{k\to\infty} \|\m{g}^{\mc{A}}(\m{x}_k)\| = 0$. Consequently, for $\mc{I} = \mc{A}(\m{x}_k)$, we obtain
\begin{align*}
\lim_{k \to \infty }  \| \m{d}_k -  \m{Z}_{\m{A}_{\mc{I}}} \m{u}_k \| = \lim_{k \to \infty }  \alpha \| \m{Z}_{\m{A}_{\mc{I}}}  \m{Z}_{\m{A}_{\mc{I}}}^\top \m{g}(\m{x}_k)\|=0
\end{align*}
for some $\alpha>0$.

Since $ \phi_k^{''}(0)= {\m{d}_k}^\top \m{H}(\m{x}_k) \m{d}_k $,  \eqref{eq:secbound} implies  $\liminf_{k\to\infty}    \m{u}_k^\top \m{H}^{\mc{A}}(\m{x}_k)  \m{u}_k  \geq 0$ which is a contradiction to \eqref{eq:negcurv}. Thus, \ref{itm:sec} holds. 
\end{proof}

The following Theorem~\ref{global_theorem} establishes the global convergence of \NPASA{} for solving the polyhedral constrained problem \eqref{P}.

\begin{theorem}\label{global_theorem}
Given Assumptions~\ref{assu-a} and \ref{assu-b}, assume that a sequence $\{\m{x}_k\}$ is generated by \NPASA{} with $\epsilon_E = 0$. Then, there exists at least one limit point of $\{\m{x}_k\}$ such that
\begin{equation}\label{eq:glob}
\liminf\limits_{k \to\infty} E(\m{x}_k)=0.
\end{equation}
Further, any limit point of $\{\m{x}_k\}$ is a stationary point of \eqref{P}.
\end{theorem}
\begin{proof}
The proof is similar to that of \cite[Theorem 3.2]{hager2016active}. By Assumption~\ref{assu-a}, there exists at least one limit point $\m{x}^\ast$ of $\m{x}_k$. Let $\{\m{x}_k:~ k \in K\}$ be an infinite subsequence of $\m{x}_k$ such that $\lim_{ k\in K, k\rightarrow \infty} \m{x}_k = \m{x}^\ast$. We consider three cases: \\
\textbf{Case~1}. \textit{Only \GP{} is performed for $k$ sufficiently large}. By Assumption~\ref{assu-a}
and \cite[Theorem 2.4]{calamai1987projected},  $\lim_{k\in K, k\rightarrow \infty} \|\m{x}_{k+1} -\m{x}_k\|/s_k = 0.$
This together with the fact that $\|\m{x}_{k+1} - \m{x}^\ast\| \leq \| \m{x}_{k+1} -\m{x}_k\| + \| \m{x}_{k} -\m{x}^\ast\|$, implies  $\lim_{k\rightarrow \infty, k\in K} \m{x}_{k+1} = \m{x}^\ast$. 

Now, since $E(\m{x})=\|\nabla_\Omega f(\m{x})\|$ is lower semicontinuity \cite[Lemma 3.3]{calamai1987projected}, it follows from \cite[Theorem 3.4]{calamai1987projected} that $$\lim_{ k\in K, k\rightarrow \infty} E(\m{x}_{k+1})  = 0.$$ 
\textbf{Case~2}. \textit{Only \NLCO{} is performed for $k$ sufficiently large}. Since $\theta$ is never reduced in phase two, there exists $\bar{\theta} > 0$ such that $\theta \equiv \bar{\theta}$ for $k$ sufficiently large. Hence, we have $e(\m{x}_k) \ge \bar{\theta} E(\m{x}_k)$ when $k$ is  large enough. We note that no index in the active set can be freed from $\m{x}_k$ to $\m{x}_{k+1}$ using the \NLCO{}. This together with \eqref{eqn:act:seq} implies that the active set becomes unchanged for $k$ large enough. Consequently, Lemma~\ref{thm:glob:face}\ref{itm:fir} together with the inequality $e(\m{x}_k) \ge \bar{\theta} E(\m{x}_k)$  gives \eqref{eq:glob}. 
\\
\textbf{Case~3}. \textit{There are an infinite number of branches from \NLCO{} to \GP{}}. In this case, the \GP{} step is performed an infinite number of iterations at $k_1, k_2, \ldots $, where $\{k_i\} \subseteq K$. Now, by the same argument as in Case~1, we have $\lim_{k_i\rightarrow \infty} \m{x}_{k_i+1} = \m{x}^\ast$,  and $$\lim_{ k_i\in K, k_i\rightarrow \infty} E(\m{x}_{k_i+1})  = 0.$$

To show the second half of the theorem, we note that by the lower semicontinuity of $E(\m{x})$,
$$E(\m{x}^\ast) \leq \liminf\limits_{k \to\infty} E(\m{x}_k) = 0,$$ which guarantees that $\m{x}^\ast$ is a stationary point of \eqref{P}.
\end{proof}

\section{Local convergence analysis}\label{sec:4}

This section presents theoretical properties to quantify the asymptotic behavior of the proposed algorithm. In particular, we show that the \NLCO{} step is taken asymptotically when the active constraints are linearly independent and a strong second-order sufficient optimality condition hold.  

Let $\m{x}^\ast$ be a stationary point of \eqref{P} and $\Lambda (\m{x}^*)$ be the set of Lagrange multipliers associated with the constraints. That is, $\m{x}^\ast \in \Omega$ and for any $\g{\lambda}^\ast \in \Lambda (\m{x}^\ast)$, $(\m{x}^\ast, \g{\lambda}^\ast)$
satisfies
\begin{align}\label{x-opt}
\nonumber
\m{g}(\m{x}^*) + \m{A}\tr \g{\lambda}^*&= \m{0}, \quad
\g{\lambda}^* \ge \m{0}, \quad \lambda_i^* = 0
\mbox{ if } i \in \mc{F}(\m{x}^*),\\
\lambda^*_i(\m{A} \m{x}^*-\m{b})_i&=0, \quad \forall i.
\end{align}
For $ s \ge 0$, let
\begin{equation}\label{y-def}
\m{y}(\m{x}, s) = \mc{P}_\Omega (\m{x} - s \m{g}(\m{x})) =
\arg \;
\min \left\{ \frac{1}{2} \|\m{x} - s \m{g}(\m{x}) - \m{y}\|^2:
\m{Ay} \le \m{b} \right\}.
\end{equation}
Let $\m{y} = \m{y}(\m{x}, s)$ and define $\Lambda (\m{x}, s)$ the set of multipliers $\g{\lambda} (\m{x})$ associated with the polyhedral constraints in \eqref{y-def}. Then, $\g{\lambda} (\m{x}) \in  \Lambda (\m{x}, s)$ is any vector that satisfies the conditions
\begin{align}\label{y-opt}
\nonumber
\m{y} - (\m{x} - s \m{g}(\m{x})) + \m{A}^\top \g{\lambda} (\m{x}) &= \m{0}, \quad
\g{\lambda} (\m{x}) \ge \m{0}, \quad \lambda_i (\m{x}) = 0
\mbox{ if } i \in \mc{F}(\m{y}),\\
\lambda_i(\m{x})(\m{A}\m{y}-\m{b})_i&=0, \quad \forall i.
\end{align}
It follows from \eqref{x-opt} and \eqref{y-opt} that 
\begin{equation}\label{scaling}
 \m{y}(\m{x}^*, s) = \m{x}^* \quad \text{and} \quad \Lambda(\m{x}^*, s) = s \Lambda (\m{x}^*).
\end{equation}

Influenced by \cite{burke1994exposing,calamai1987projected,hager2006new,hager2016active,zhang2020smoothing}, we next show that \NPASA{} is able to identify the active variables that are nondegenerate at the solution in a finite number of iterations. This is essential in providing the quadratic convergence result of our proposed active set algorithm.
\begin{lem}\label{thm:iden}
If \NPASA{} with $\epsilon_E = 0$ generates an infinite sequence of iterates $\{\m{x}_k\}$ converging to $\m{x}^*$ where the active constraint gradients are linearly independent, then there is an index $K$ such that $\mc{A}_+(\m{x}^*)  \subset  \mc{A}(\m{x}_k)$ for all $k \geq K$ where 
\begin{align*}
\mc{A}_+(\m{x}^*) = \{i \in \mc{A}(\m{x}^*): \Lambda_i (\m{x}^*) > 0\}.
\end{align*}
\end{lem}
\begin{proof}
Since the rows of $\m{A}$ corresponding to indices $i \in \mc{A}(\m{x}^*)$ are linearly independent, $\Lambda (\m{x}^*)$ is a singleton. By \eqref{scaling}, $\Lambda (\m{x}^*, \alpha)$ is also singleton for any $s>0$. Since $\{\m{x}_k\}$ converges to $\m{x}^*$, and $\m{y}(\m{x}, .)$ is a Lipschitz continuous function, there exists an integer $\bar{K}$ such that $ \mc{A}(\m{y}(\m{x}^k, s )) \subset \mc{A} (\m{x}^\ast)$ for all $k \geq \bar{K}$. Consequently, the gradients of the active constraints at $\m{y}(\m{x}_k, s )$ are linearly independent which implies that $\Lambda(\m{x}_k, s)$ is a singleton for $k \geq \bar{K}$. Then, from \cite[Corollary 6.1]{hager2016active}, we get
\begin{equation}\label{eqn:ste:lip}
 \|\Lambda(\m{x}_k, s) - \Lambda (\m{x}^*, s) \|
\le c_3\|\m{x}_k - \m{x}^*\|
\end{equation}
for some constant $c_3$.

For any $i \in \mc{A}_{+}(\m{x}^\ast)$, the Lagrange multiplier $\lambda^* \in \Lambda (\m{x}^*)$ satisfies $\lambda^*_{i} > 0$. This together with \eqref{eqn:ste:lip} implies that there exists an integer $\hat{K}_{i}$ such that  $\lambda_{ki} > 0$, $ \g{\lambda}_{k} \in \Lambda(\m{x}_k, s)$ for $k \geq \hat{K}_{i}$. From \eqref{y-opt} we get $(\m{A} \m{y}(\m{x}_k, s ) - \m{b})_{i}=0$ which implies that $i \in \mc{A}(\m{y}(\m{x}_k, s ))$. Now, let $\hat{{K}} = \max\big\{ \hat{K}_{i}, ~~ i \in  \mc{A}_+(\m{x}^*)\big\}$ and $\tilde{K} = \max \big\{ \bar{K}, \hat{K} \big\}$. Clearly, for any $i \in \mc{A}_{+}(\m{x}^*)$, and any given $s >0$, we have
\begin{equation}\label{y-active}
i \in \mc{A}(\m{y}(\m{x}_k, s)) , \qquad \textnormal{for all} \quad  k \geq  \tilde{K}.
\end{equation}

We now consider two disjoint cases for the iterates generated by Algorithm~\ref{algo:npasa}:\\
\textbf{Case~1}. \textit{There exists an integer $\tilde{K}_1 \geq \tilde{K}$ such that $\m{x}_{\tilde{K}_1+1}$ is obtained from \GP{}}. Then, for any $k \geq \tilde{K}_1$ such that $\m{x}_{k+1}$ is obtained from $\m{x}_k$ by the \GP{} step, it follows from \eqref{y-def} that
\begin{equation}\label{eqn:iden2}
\m{x}_{k+1} = \mc{P}_{\Omega} (\m{x}_k - s_k \m{g}(\m{x}_k))= \m{y}(\m{x}_k, s_k).
\end{equation}
This together with \eqref{y-active} implies that $i \in \mc{A} (\m{x}_{k+1})$ for any $i \in \mc{A}_+(\m{x}^*)$. Since no active constraint can be freed by the NLCO step in phase two, we get
$
i \in \mc{A} (\m{x}_k)$ for all $  k \geq \tilde{K}_1 + 1$.
\\
\textbf{Case~2}.
\textit{$\m{x}_{k+1}$ is obtained from \NLCO{} for any $k \geq \tilde{K}$}. By \eqref{eqn:act:seq}, the active constraints become stable after a finite number of steps which implies that there exists an integer $\tilde{K}_2 > \tilde{K}$ such that
\begin{equation}\label{eqn:iden3}
\mc{A}(\m{x}_k) \equiv \mc{I}  \subseteq   \mc{A}(\m{x}^*),  \qquad \textnormal{for all} \quad  k \geq \tilde{K}_2.
\end{equation}

Now, consider the following optimization problem 
\begin{equation}\label{eq:proj:prob}
\min_{\m{y}} \frac{1}{2}\|  \m{g}(\m{x}_k) -\m{y}\|^2  \quad  \text{s.t.} \quad \m{a}^\top_i \m{y} =0 ~\mbox{ if }~ i \in \mc{I}
\end{equation}
From the first-order optimality conditions, there exists a vector $\g{\mu}_k \in \mathbb{R}^m$ such that
\begin{equation}\label{eqn:iden4}
\m{g}(\m{x}_k) -  \m{Z}_{\m{A}_{\mc{I}}} \m{Z}_{\m{A}_{\mc{I}}}^\top \m{g}(\m{x}_k)  + \m{A}^\top \g{\mu}_{k} = \m{0}, \quad
\g{\mu}_k \ge \m{0}, \quad (\g{\mu}_k)_{i} = 0 ~\mbox{ if }~ i \in \mc{I}^c.
\end{equation}
We note that $\g{\mu}_k$ is unique since the column vectors $\m{a}_i$, $ i \in \mc{I} \subseteq \mc{A}(\m{x}^\ast)$ are linearly independent. Further, by the strong convexity of objective in \eqref{eq:proj:prob} with $\m{x}_k$ being replaced by $\m{x}^\ast$, and the linear independence of $\{ \m{a}_i, i \in \mc{A}(\m{x}^\ast)\}$, there exists a unique vector $ \bar{\m{g}}^\ast\in \mathbb{R}^n$ and a unique vector $\bar{\g{\lambda}}^\ast \in \mathbb{R}^m$ such that
\begin{equation}\label{eqn:iden5}
\m{g}(\m{x}^\ast)- \bar{\m{g}}^\ast + \m{A}^\top \bar{\g{\lambda}}^\ast= \m{0}, \quad
\bar{\g{\lambda}}^\ast \ge \m{0}, \quad \bar{\lambda}^\ast_i = 0
\mbox{ if } i \in \mc{F}(\m{x}^\ast).
\end{equation}

It follows from the first-order optimality conditions of \eqref{P} given in \eqref{x-opt} that 
\begin{equation}\label{eqn:iden6}
\m{g}(\m{x}^*) +\m{A}^\top \g{\lambda}^\ast = \m{0}, \quad
\g{\lambda}^\ast \ge \m{0}, \quad \lambda_i^\ast  = 0~\mbox{ if}~i \in \mc{F}(\m{x}^\ast).
\end{equation}
Now, we get $\bar{\m{g}}^\ast=0$ and $\bar{\g{\lambda}}^\ast= \g{\lambda}^\ast$, by comparing \eqref{eqn:iden5} with \eqref{eqn:iden6} and using the uniqueness of $\bar{\m{g}}^\ast$ and $\bar{\g{\lambda}}^\ast$ in \eqref{eqn:iden5}.  By Lemma~\ref{global_theorem},  $\lim_{k_j \rightarrow \infty}  \m{Z}_{\m{A}_{\mc{I}}} \m{Z}_{\m{A}_{\mc{I}}}^\top \m{g}(\m{x}_{k_j})=0$ for some infinite subsequence $\{k_j\} \subset \{k\}$ which together with \eqref{eqn:iden4} gives
\begin{eqnarray}\label{eqn:iden7}
\nonumber
0&=&\lim_{k_j \rightarrow \infty}  \m{g}_{k_j} + \lim_{k_j \rightarrow \infty} \m{A}^\top \g{\mu}_{k_j} \\
&=& \m{g}(\m{x}^*)  + \lim_{k_j \rightarrow \infty} \m{A}^\top \g{\mu}_{k_j} , \quad (\g{\mu}_{k_j})_i=0~\mbox{ if }~i \in \mc{I}^c.
\end{eqnarray}
Comparing \eqref{eqn:iden6} and \eqref{eqn:iden7}, and noting the uniqueness of $\g{\lambda}^\ast$ in \eqref{eqn:iden6}, we obtain $\lim_{k_j \rightarrow \infty}  (\g{\mu}_{k_j})_i= \g{\lambda}^*_i$ if $ i \in \mc{A}_+(\m{x}^\ast)$. By \eqref{eqn:iden4}, $(\g{\mu}_{k})_i = 0$ if $i \in \mc{I}^c$ which implies that $\lim_{k_j \rightarrow \infty}  (\g{\mu}_{k_j})_i= 0$ if  $i \in \mc{A}_+(\m{x}^\ast) \setminus \mc{I}$ and as a result  $\mc{A}_+(\m{x}^\ast) \setminus \mc{I} = \emptyset$. This shows that for any $i \in \mc{A}_{+}(\m{x}^\ast)$ and $k \geq \widehat{K}_2 $, we have $i \in \mc{I}=\mc{A}(\m{x}_k)$.

Taking both cases into consideration, there exists an index $K$ such that $\mc{A}_+(\m{x}^*)  \subset  \mc{A}(\m{x}_k) $.
\end{proof}

Next, we show  that \NPASA{} takes the \NLCO{} step asymptotically when the active constraints are linearly independent and a strong second-order sufficient optimality condition hold. We say that a stationary point $\m{x}^*$ of (\ref{P}) satisfies the strong second-order sufficient optimality condition if there exists
$\sigma_H> 0$ such that
\begin{equation} \label{opt2}
\m{d} \tr \nabla^2 f(\m{x}^*) \m{d} \ge \sigma_H\|\m{d}\|^2  \quad
\textnormal{whenever} \quad (\m{Ad})_i = 0 \textnormal{ for all } i \in \mc{A}_+(\m{x}^*).
\end{equation}
\begin{theorem}\label{thm:swi}
If \NPASA{} with $\epsilon_E = 0$ generates an infinite sequence of iterates $\m{x}_k$ converging to $\m{x}^*$ where the active constraint gradients are linearly independent and the strong second-order sufficient optimality condition holds, then there exists $\theta^* > 0$ such that $
e(\m{x}_k) \ge \theta^* E(\m{x}_k)$ for $k$ sufficiently large.
\end{theorem}

\begin{proof}
Lemma~\ref{thm:iden} guarantees that  $\mc{A}_+(\m{x}^*)  \subset  \mc{A}(\m{x}_k)$  for all $k$ sufficiently large. This together with \cite[Theorem 4.5]{burke1994exposing} implies $\lim_{k\rightarrow \infty} \mc{P}_{\mc{T}(\m{x}_k)}[- \m{g}(\m{x}_k)]=0$, where  $\mc{T}(\m{x}_k)$ is the tangent cone defined in \eqref{eqn:swi5}.  Now, we have 
\begin{eqnarray}\label{eqn:swi4}
\nonumber
\lim_{k\rightarrow \infty} \left\| \mc{P}_{\mc{T}(\m{x}_k)}[-\m{g}(\m{x}^\ast)] \right\| &=& \lim_{k\rightarrow \infty} \left\|\mc{P}_{\mc{T}(\m{x}_k)}[- \m{g}(\m{x}_k)] - \mc{P}_{\mc{T}(\m{x}_k)} [ -  \m{g}(\m{x}_k)] + \mc{P}_{\mc{T}(\m{x}_k)}[- \m{g}(\m{x}^\ast)]\right\|\\
 \nonumber
 & \leq & \lim_{k\rightarrow \infty} \left\|\m{g}(\m{x}^*)-\m{g}(\m{x}_k )\right\| + \left\|\mc{P}_{\mc{T}(\m{x}_k)}[- \m{g}(\m{x}_k)] \right\| \\
 & =0.
\end{eqnarray}
Here, the  inequality is obtained from  the non-expensive property of the projection operator $\mc{P}$; and the equality follows from  our assumption that $\{\m{x}_k\}$ converges to $\m{x}^\ast$ and continuity of $\m{g} (\m{x})$.

Now, choose $k$ large enough such that $\m{x}_k$ is sufficiently close to $\m{x}^\ast$. Then, we have  $\mc{F}(\m{x}^\ast) \subseteq \mc{F}(\m{x}_k)$ and by Lemma~\ref{thm:iden} there exists an integer $K$ such that
\begin{equation}\label{eqn:swi6}
\mc{A}_+(x^*) \subseteq   \mc{A}(\m{x}_k) \subseteq   \mc{A}(x^*)    \qquad \textnormal{for all} \quad  k \geq K.
\end{equation}
From \eqref{eqn:swi6} and the fact that $\mc{A}(\m{x}^\ast)$ has a finite number of subsets, there are only a finite number of index sets $\mc{A}_1, \ldots, \mc{A}_J$ for $\mc{A}(\m{x}_k)$, $k = 1, 2, \ldots$ . Let
\begin{equation*}
\mc{T}_j = \{ \m{d}~:~ \m{a}_i^\top \m{d} \leq 0~~  i \in \mc{A}_j\}, \quad \text{for all} \quad j=1,2, \ldots, J.
\end{equation*}
Suppose
$
\{\mc{T}_1, \mc{T}_2, \ldots, \mc{T}_{\bar{J}} \}  \subseteq  \{\mc{T}_1, \mc{T}_2, \ldots, \mc{T}_{J}\}
$
is composed of all the elements in $\{\mc{T}_1,\mc{T}_2, \ldots, \mc{T}_J\}$ such that each $\mc{T}_j , j = 1, 2, \ldots , \bar{J}$, includes an infinite number of $\mc{T}(\m{x}_k)$, $k \geq 1$. From \eqref{eqn:swi4}, we have $\mc{P}_{\mc{T}_j}[ - \m{g}(\m{x}^\ast)] = 0$ for $j = 1, 2, \ldots, \bar{J}$. Consequently, for all $k$ sufficiently large, we get
\begin{equation*}
\mc{P}_{\mc{T}(\m{x}_k)} [-\m{g}(\m{x}^\ast)] \in \left\{\mc{P}_{\mc{T}_1}[-\m{g}(\m{x}^\ast)], \mc{P}_{\mc{T}_2}[-\m{g}(\m{x}^\ast), \ldots, \mc{P}_{\mc{T}_{\bar{J}}}[-\m{g}(\m{x}^\ast)] \right\},
\end{equation*}
which implies
\begin{equation}\label{eqn:swi7}
\mc{P}_{\mc{T}(\m{x}_k)} [-\m{g}(\m{x}^\ast)] =0
\end{equation}
for $k$ sufficiently large. Since $E(\m{x}_k) = \| \mc{P}_{\mc{T}(\m{x}_k)}
[- \m{g}(\m{x}_k)] \|$, we obtain 
\begin{eqnarray}\label{eqn:ppart1}
\nonumber
E(\m{x}_k) &=& \| \mc{P}_{\mc{T}(\m{x}_k)}
[- \m{g}(\m{x}_k)] - \mc{P}_{\mc{T}(\m{x}_k)} [ - \m{g}(\m{x}^\ast)] + \mc{P}_{\mc{T}(\m{x}_k)
}[- \m{g}(\m{x}^\ast)]\| \\
\nonumber
& \leq & \| \m{g}(\m{x}_k) - \m{g}(\m{x}^\ast)\| + \|\mc{P}_{\mc{T}(\m{x}_k)}[-\m{g}(\m{x}^\ast)]\|\\
& \leq &  U_H \| \m{x}_k - \m{x}^\ast\|,
\end{eqnarray}
where the second inequality follows from \eqref{eqn:swi7} and the Lipschitz continuity of $\m{g}$ with the constant $U_H$.

By the continuity of the second derivative of $f$, it follows from (\ref{opt2}) that for $\rho> 0$ sufficiently small,
\begin{eqnarray}
\quad \quad \quad (\m{x}-\m{x}^*)\tr (\m{g} (\m{x})-\m{g} (\m{x}^*)) &=&
(\m{x}-\m{x}^*)\tr \int_0^1 \nabla^2 f (\m{x}^* + t (\m{x} - \m{x}^*)) dt \;
(\m{x} - \m{x}^*) \nonumber \\
&\ge& 0.5\sigma_H \|\m{x}-\m{x}^*\|^2
\label{31}
\end{eqnarray}
for all $\m{x} \in \mc{B}_{\rho}(\m{x}^*) \cap \mc{S}_+$, where
\[
\mc{S}_+ = \{ \m{x} \in \mathbb{R}^n :
(\m{Ax}- \m{b})_i = 0 \mbox{ for all } i \in \mc{A}_+(\m{x}^*)\} .
\]
Choose $k$ large enough such that $\m{x}_k \in \mc{B}_\rho (\m{x}^*)$. The bound (\ref{31}) yields
\begin{equation} \label{upper}
0.5\sigma_H\|\m{x}_k-\m{x}^*\|^2 \le
(\m{x}_k - \m{x}^*)\tr (\m{g}(\m{x}_k) - \m{g}(\m{x}^*)) .
\end{equation}
By the first-order optimality conditions for a local minimizer $\m{x}^*$ of
(\ref{P}), there exists a multiplier $\g{\lambda}^* \in \mathbb{R}^m$ such that
\begin{equation}\label{1st}
\m{g}(\m{x}^*) + \m{A}\tr \g{\lambda}^* = 0 \quad \mbox{where} \quad
\lambda^*_i(\m{A} \m{x}^*-\m{b})_i=0, \quad \forall i \; \mbox{ and } \;
\g{\lambda}^* \ge \m{0}.
\end{equation}
For $k$ large enough, we have $\mc{A}_+(\m{x}^*) \subseteq  \mc{A}(\m{x}_k)$ and  $\lambda_i^* [\m{A}(\m{x}_k - \m{x}^*)]_i = 0$
for each $i$ since $[\m{A}(\m{x}_k - \m{x}^*)]_i = 0$
when $i \in \mc{A}_+(\m{x}^*)$ and $\lambda_i^* = 0$ when $i \in \mc{A}_+(\m{x}^*)^c$.
Hence, we have
\[
[\m{A}(\m{x}_k - \m{x}^*)]_i  \g{\lambda}^*_i = 0, \quad \forall i .
\]
Thus,
\begin{equation}\label{x1}
(\m{x}_k - \m{x}^*)\tr \m{g}(\m{x}^*) =
(\m{x}_k - \m{x}^*)\tr (\m{g}(\m{x}^*) + \m{A}\tr \g{\lambda}^*) = \m{0}
\end{equation}
by the first equality in (\ref{1st}).

Let $\mc{I}=\mc{A}(\m{x}_k)$ and  $ \bar{\m{g}}_k = \m{Z}_{\m{A}_{\mc{I}}} \m{Z}_{\m{A}_{\mc{I}}}^\top \m{g}(\m{x}_k)$. The first-order optimality conditions for the minimizer $\bar{\m{g}}_k$ in \eqref{eq:proj:prob} implies the existence of $\g{\lambda}_\mc{I}$ such that
\begin{equation}\label{2nd}
\bar{\m{g}}_k -
\m{g}(\m{x}_k) + \m{A}_\mc{I}\tr \g{\lambda}_\mc{I} = \m{0}
\quad \mbox{where} \quad \m{A}_\mc{I} \m{x}_k = \m{b}_\mc{I} .
\end{equation}
Since $\mc{A}(\m{x}_k) \subset \mc{A}(\m{x}^*)$, we have $(\m{A} (\m{x}_k - \m{x}^*))_i = 0$ for $i \in \mc{A}(\m{x}_k)$, and
\begin{equation}\label{x02}
 (\m{A}_\mc{I} (\m{x}_k - \m{x}^*))^\top \g{\lambda}_\mc{I} = 0.
 \end{equation}
Now, it follows from \eqref{2nd} and \eqref{x02} that
\begin{equation}\label{x2}
(\m{x}_k - \m{x}^*)\tr \m{g}(\m{x}_k) =
(\m{x}_k - \m{x}^*)\tr (\m{g}(\m{x}_k) -
\m{A}_\mc{I}\tr \g{\lambda}_\mc{I}) =
(\m{x}_k - \m{x}^*)\tr
\bar{\m{g}}_k.
\end{equation}

Combine (\ref{upper}), (\ref{x1}), and (\ref{x2}) to obtain
\begin{equation}\label{dbar}
0.5\sigma_H\|\m{x}_k-\m{x}^*\| \le \|\bar{\m{g}}_k\| =\|\m{g}^{\mc{I}}(\m{x}_k)\|= e(\m{x}_k).
\end{equation}

Finally, we combine \eqref{eqn:ppart1} with \eqref{dbar} to get $
e(\m{x}_k) \ge \theta^* E(\m{x}_k)$ for $k$ sufficiently large where $\theta^*=\frac{\sigma_H}{2U_H}$.
\end{proof}

The following corollary shows that \NPASA{} only perform the \NLCO{} within a finite number of iterations.

\begin{cor}\label{cor:iwe}
If \NPASA{} with $\epsilon_E= 0$ generates an infinite sequence of iterates
converging to  $\m{x}^*$ where the active constraint gradients are linearly independent and the strong second-order sufficient optimality condition holds,
then within a finite number of iterations, only \NLCO{} is executed.
\end{cor}

\begin{proof}
First, we claim that the \NLCO{} step is taken within a finite number of iterations.
By Theorem~\ref{thm:swi}, there exists $\theta^* > 0$ such that $e(\m{x}_k) \ge \theta^* E(\m{x}_k)$. By contrary assume only \GP{} is executed, then $\theta$ is decreased in each iteration, and will be decreased to $\theta < \theta^*$ after a finite number of iterations which implies that phase one branches to phase two and phase two cannot branch to phase one. This is a contradiction. 

When the \NLCO{} step is taken, phase two cannot branch to phase one an infinite number of times. Otherwise, $\theta $ will be reduced to $\theta < \theta^*$ and  $e(\m{x}_k) \ge \theta E(\m{x}_k)$ will occur.
\end{proof}

\subsection{Convergence rate of \NPASA{}}

In the previous section, we have derived the asymptotic behavior of \NPASA{} and shown that \NPASA{} takes the second order step after finite number of iterations. We now aim to show the local quadratic convergence of \NPASA{} to a stationary point of \eqref{P}. 
\begin{theorem}\label{thm:rate:noncon}
Suppose $\m{x}^*$ is a stationary point where the active constraint gradients are linearly independent and the strong second-order sufficient optimality condition holds. Assume \NPASA{} with $\epsilon_E=0$ generates an infinite sequence of iterates $\m{x}_k$ converging to $\m{x}^*$ and the second-derivative matrix is Lipschitz continuous in a neighborhood of $\m{x}^*$, then
\begin{enumerate}[label=(\roman*)]
\item  \label{itm:rate:fir} 
 the sequence of projected gradient norms converges quadratically to zero; and
\item  \label{itm:rate:sec}  the rate of convergence of $\{\m{x}_k\}$ is quadratic.
\end{enumerate}
\end{theorem}
\begin{proof}

\ref{itm:rate:fir}. It follows from Corollary~\ref{cor:iwe} that within a finite number of iterations, only  \NLCO{} is executed. Let $\mc{I}=\mc{A}(\m{x}_k)$ and $\mc{S}=\mc{A}_+(\m{x}^*)$. By Lemma~\ref{thm:iden}, we have  $ \mc{S} \subset \mc{I}$ which implies that the range space of $\m{Z}_{\m{A}_{\mc{I}}}$ is contained in the range space of $\m{Z}_{\m{A}_{\mc{S}}}$.  This, together with the strong second-order sufficient optimality condition \eqref{opt2} implies that $
\m{H}^{\mc{I}}(\m{x}^*) = \m{Z}_{\m{A}_{\mc{I}}}^\top \m{H}(\m{x}^*)\m{Z}_{\m{A}_{\mc{I}}}
$ is positive definite for all $k$ sufficiently large. From the continuity of $ \nabla^2f(\m{x})$, it follows that the reduced Hessian $ \m{H}^{\mc{I}}(\m{x}_k) = \m{Z}_{\m{A}_{\mc{I}}}^\top \m{H}(\m{x}_k)\m{Z}_{\m{A}_{\mc{I}}}$ is positive definite with smallest eigenvalue bounded below by $\sigma_H > 0$. Thus, Algorithm~\ref{algo:npasa} computes the Newton direction and the unit step length is admissible for $k$ large enough \cite{more1982newton}. 

Now, from \eqref{eq:newton}, we obtain
\begin{eqnarray}\label{eq:UA:b2b}
\nonumber
\|\m{d}_k\| &=&  \| \m{Z}_{\m{A}_{\mc{I}}} \big(\m{Z}_{\m{A}_{\mc{I}}}^\top  \m{H}(\m{x}_k) \m{Z}_{\m{A}_{\mc{I}}}\big)^{-1} \m{Z}_{\m{A}_{\mc{I}}}^\top \m{g}(\m{x}_k) \| \\ 
 &\leq &  \frac{1}{\sigma_H}  \|\m{Z}_{\m{A}_{\mc{I}}} \m{Z}_{\m{A}_{\mc{I}}}^\top \m{g}(\m{x}_k)\| = \frac{1}{\sigma_H} \|\m{g}^{\mc{I}}(\m{x}_k)\|,
\end{eqnarray}
where the inequality follows since the columns of $\m{Z}_{\m{A}_{\mc{I}}}$ form an orthonormal basis for the subspace $\mc{N}(\m{A}_{\mc{I}})$ and $\sigma_{\min}(\m{Z}_{\m{A}_{\mc{I}}}^\top  \m{H}(\m{x}_k) \m{Z}_{\m{A}_{\mc{I}}}) \geq \sigma_H>0$.

Let $L_H$ denote the Lipschitz constant of $\nabla^2 f(\m{x})$ for $\m{x}$ near $\m{x}^*$. Choose $k$ large enough such that $\m{x}_k$ is in a neighborhood of $\m{x}^*$. From the Lipschitz continuity of  $\nabla^2 f(\m{x})$, we have
\begin{eqnarray}\label{eq:super}
\nonumber
\|\m{g}^{\mc{I}}(\m{x}_{k+1}) \| &=& \|\m{Z}_{\m{A}_{\mc{I}}}^\top(\m{g}(\m{x}_{k+1})- \m{g}(\m{x}_k) - \m{H}(\m{x}_k) \m{d}_k)\|\\
 \nonumber
& \leq & \frac{L_H}{2}  \|\m{d}_k \|^2 \\  
&\leq & \frac{L_H}{2\sigma_H^2}  \|\m{g}^{\mc{I}}(\m{x}_k)\|^2, 
\end{eqnarray} 
where the second inequality follows from \eqref{eq:UA:b2b}.

In Algorithm~\ref{algo:nlco}, we have $\mc{I}=\mc{A}(\m{x}_{k})  \subset \mc{A}(\m{x}_{k+1})=\mc{J}$ which implies that
\begin{eqnarray}\label{eq:UA:b22}
 \|\m{g}^{\mc{J}}(\m{x}_{k+1}) \| \leq \|\m{g}^{\mc{I}}(\m{x}_{k+1})\|.
\end{eqnarray}
This inequality combined with \eqref{eq:super} yields
\begin{eqnarray}\label{eq:quad:e}
 \|\m{g}^{\mc{A}}(\m{x}_{k+1}) \| = \|\m{g}^{\mc{J}}(\m{x}_{k+1}) \|  \leq  \frac{L_H }{2\sigma_H^2} \|\m{g}^{\mc{I}}(\m{x}_k)\|^2 = \frac{L_H }{2\sigma_H^2}   \|\m{g}^{\mc{A}}(\m{x}_{k}) \|^2.
\end{eqnarray}
This essentially implies that the sequence of projected gradient norms converges quadratically to zero.

\ref{itm:rate:sec}. From \eqref{dbar} for $k$ sufficiently large, we have
\begin{eqnarray}\label{eq:gx}
 \|\m{x}_{k+1} - \m{x}^*\|  \leq  \frac{2}{\sigma_H}  \|\m{g}^{\mc{A}}(\m{x}_{k+1}) \|  \leq \frac{L_H}{\sigma_H^3 }  \|\m{g}^{\mc{A}}(\m{x}_{k}) \|^2,
\end{eqnarray}
where the second inequality follows from \eqref{eq:quad:e}.

By Lemma~\ref{thm:iden}, we have  $\mc{S}=\mc{A}_+(\m{x}^*)  \subset  \mc{A}(\m{x}_k)$ which implies that 
\begin{eqnarray}\label{eqn:sa}
\|\m{g}^{\mc{A}}(\m{x}_{k}) \| \leq  \|\m{g}^{\mc{S}}(\m{x}_{k}) \|.
\end{eqnarray}
Combine \eqref{eq:gx} with \eqref{eqn:sa} to get
\begin{eqnarray}\label{eqn:bound:num}
 \|\m{x}_{k+1} - \m{x}^*\| \leq \frac{L_H}{\sigma_H^3} \|\m{g}^{\mc{S}}(\m{x}_{k}) \|^2.
\end{eqnarray}
Since $\m{g}^{\mc{S}}(\m{x}^*) =0$ and the columns of $\m{Z}_{\m{A}_{\mc{S}}}$ form an orthonormal basis for the subspace $\mc{N}(\m{A}_{\mc{S}})$, we have
\begin{eqnarray} \label{eqn:bound:denum}
\nonumber
\|\m{g}^{\mc{S}}(\m{x}_{k}) \| &\leq&  \| \m{Z}_{\m{A}_{\mc{S}}}^\top (\m{H}(\m{x}^*)(\m{x}_k-\m{x}^*))\| + \| \m{Z}_{\m{A}_{\mc{S}}}^\top  (\m{g}(\m{x}_{k})- \m{g}(\m{x}^*) - \m{H}(\m{x}^*) (\m{x}_k-\m{x}^*))\| \\
\nonumber
& \leq &   \| \m{Z}_{\m{A}_{\mc{S}}}^\top  (\m{H}(\m{x}^*)(\m{x}_k-\m{x}^*))\|  + \frac{L_H}{2} \|\m{Z}_{\m{A}_{\mc{S}}}^\top (\m{x}_k-\m{x}^*)\|  \\
\nonumber
& = &   \|(\m{Z}_{\m{A}_{\mc{S}}}^\top \m{H}(\m{x}^*) \m{Z}_{\m{A}_{\mc{S}}})  \m{Z}_{\m{A}_{\mc{S}}}^\top  (\m{x}_k-\m{x}^*)\|  + \frac{L_H}{2} \|\m{Z}_{\m{A}_{\mc{S}}}^\top(\m{x}_k-\m{x}^*)\|  \\
&\leq &  (0.5L_H + \kappa_H ) \|\m{x}_k-\m{x}^*\|,
\end{eqnarray}
where the third inequality follows from the Lipschitz continuity of $\m{H}(\m{x})$ in a neighborhood of $\m{x}^*$; and the last inequality follows since $\m{Z}_{\m{A}_{\mc{S}}}  \m{Z}_{\m{A}_{\mc{S}}}^\top  (\m{x}_k-\m{x}^*) =\m{x}_k-\m{x}^*$ and 
\begin{equation*}
\kappa_H := \|\m{Z}_{\m{A}_{\mc{S}}}^\top \m{H}(\m{x}^*) \m{Z}_{\m{A}_{\mc{S}}}\|= \| \m{H}^{\mc{S}}(\m{x}^*)\|.
\end{equation*}

Now, we combine \eqref{eqn:bound:denum} with \eqref{eqn:bound:num} to get
\begin{equation*}
\frac{\|\m{x}_{k+1} -\m{x}^*\|}{ \|\m{x}_k-\m{x}^*\|^2} \leq M,
\end{equation*}
where
$$
 M= \frac{L_H (0.5L_H +  \kappa_H )^2}{	\sigma_H^3}.
$$
Thus, the sequence $\{\m{x}_k\}$ generated by \NPASA{} converges quadratically to $\m{x}^*$.
\end{proof}

\subsection{Convergence rate of inexact NPASA}

In Section~\ref{sec:2}, we have assumed that certain linear-algebra operations in Algorithm~\ref{algo:npasa}--the linear system solves of \eqref{eq:newton} and the computation of the negative curvature direction are performed exactly. In a large-scale setting, the cost of these operations can be prohibitive, so iterative techniques that perform these operations inexactly are of interest. Therefore, we focus our attention on iterative methods such as the conjugate gradient (CG). The CG method is the most popular iterative method for linear systems due to its rich convergence theory and strong practical performance. It has also been popular in the context of nonconvex smooth minimization \cite{nocedal2006numerical,curtis2019exploiting,royer2020newton}. It requires only matrix-vector operations involving the coefficient matrix (often these can be found or approximated without explicit knowledge of the matrix) together with some vector operations. 

\begin{algorithm}[t]
\caption{Newton-CG Linearly Constrained Optimizer (\NLCOCG)}\label{algo:nlcocg}
\begin{algorithmic}
\STATE \textbf{Parameters}: 
$  \epsilon_H \in [0, \infty);  \epsilon_R \in (0, \infty)$; $\eta,  \delta_1 \in (0, 1)$; and $\delta_2 \in [\delta_1, 1)$.  
\FOR{$k=1, 2 , \ldots $}
\STATE $\m{x}_k \in \Omega$ and $\mc{I} = \mc{A}(\m{x}_k)$;
\STATE Evaluate the reduced gradient $\m{g}^{\mc{I}}(\m{x}_k) $ and the reduced Hessian $\m{H}^{\mc{I}} (\m{x}_k)$;
\STATE  Apply the CG algorithm to find an
approximation solution $\m{p}_k $ to 
\begin{align}\label{eq:cgnewton}
\m{H}^{\mc{A}}(\m{x}_k) \m{p}_k = -\m{g}^{\mc{A}}(\m{x}_k);
\end{align}
Set $\m{d}_k= \m{Z}_{\m{A}_{\mc{I}}} \m{p}_k$;
\STATE Find the stepsize $s_k$ such that \eqref{eq:line1} and \eqref{eq:line2} hold;
\STATE Update $\m{x}_{k+1}  = \m{x}_k+ s_k \m{d}_k$;
\ENDFOR
\end{algorithmic}
\end{algorithm}

We consider a CG variant of \NPASA{} (called \NPASACG{}) which employs either the iteration of the \GP{} algorithm or the iteration of Algorithm~\ref{algo:nlcocg} (\NLCOCG{}) by given rules. The \NLCOCG{} employs the CG method for finding an approximate solution of the reduced system 
\eqref{eq:cgnewton}.

If $\m{H}^{\mc{A}}(\m{x}_k) \succ 0$, then it is well known that the CG method solves \eqref{eq:cgnewton} in at most $n-|\mc{A}(\m{x}_k)|$ iterations (the dimension of the reduced Hessian matrix). However, if $\m{H}^{\mc{A}}(\m{x}_k) \not \succ 0$, then the method might encounter a direction of nonpositive curvature, say $\m{u}_k$, such that $\m{u}_k^\top \m{H}^{\mc{A}}(\m{x}_k) \m{u}_k \leq 0$. If this occurs, then we terminate CG immediately and set $\m{d}_k = \m{Z}_{\m{A}_{\mc{I}}}\m{u}_k$ where $\mc{I}= \mc{A}(\m{x}_k)$. This choice is made rather than spend any extra computational effort attempting to approximate an eigenvector corresponding to the left-most eigenvalue of $\m{H}^{\mc{A}}(\m{x}_k)$. Otherwise, if no such direction of nonpositive curvature is encountered, then the algorithm sets $ \m{d}_k = \m{Z}_{\m{A}_{\mc{I}}} \m{p}_k$, where $ \m{p}_k$ is the final CG iterate computed prior to termination.   

Next, we provide the local convergence of \NPASACG{} to a stationary point $\m{x}^*$ where the strong second-order sufficient optimality condition holds. We require that the search direction $\m{d}_k$ satisfies
\begin{equation}\label{eqn:cg:cond}
\|\m{H}^{\mc{A}}(\m{x}_k) \m{p}_k +\m{g}^{\mc{A}}(\m{x}_k)\| \leq \xi_k \| \m{g}^{\mc{A}}(\m{x}_k)\|
\end{equation}
for some sequence $\xi_k$.

\begin{theorem}\label{thm:rate:npasacg}
Suppose $\m{x}^*$ is a stationary point where the active constraint gradients are linearly independent and the strong second-order sufficient optimality condition holds. Assume \NPASACG{} with $\epsilon_E=0$ generates an infinite sequence of iterates $\m{x}_k$ converging to $\m{x}^*$. Let the search direction $\m{d}_k$ be calculated by the CG method outlined above and \eqref{eqn:cg:cond} hold for some $\xi_k$ converging to $\xi^*$. Then,

\begin{enumerate}[label=(\roman*)]
\item  \label{itm:ratein:fir} 
 \NPASACG{} with unit step lengths converges linearly to $\m{x}^*$ 
 when $\xi^*$ is sufficiently small.
\item  \label{itm:ratein:sec} If $\xi^*=0$, \NPASACG{} with unit step lengths converges superlinearly to $\m{x}^*$.
\end{enumerate}
\end{theorem}
\begin{proof}
 Let $\mc{I}=\mc{A}(\m{x}_k)$. The main estimate needed for the rate of convergence result is obtained by noting that
\begin{eqnarray*}
\|\m{g}^{\mc{I}}(\m{x}_{k+1})\| &=& \|\m{Z}_{\m{A}_{\mc{I}}}^\top \m{g}(\m{x}_{k+1}) \| \\ 
&\leq &   \|\m{Z}_{\m{A}_{\mc{I}}}^\top\big(\m{g}(\m{x}_k) + \m{H}(\m{x}_k)\m{d}_k \big)\|  + \|\m{Z}_{\m{A}_{\mc{I}}}  ^\top\big(\m{g}(\m{x}_{k+1})- \m{g}(\m{x}_k) - \m{H}(\m{x}_k) \m{d}_k\big) \|.
\end{eqnarray*}
Assumption \eqref{eqn:cg:cond} on the step, and standard bounds yield that
\begin{eqnarray}\label{cgeq:super}
 \Big\|\m{g}^{\mc{I}}(\m{x}_{k+1}) \Big\| &\leq& \xi_k  \|\m{g}^{\mc{I}}(\m{x}_k)\| + \tau_k \|\m{d}_k\|
\end{eqnarray}
for some sequence $\tau_k$ converging to zero.    

By the same argument as in the proof of Theorem~\ref{thm:rate:noncon}, we obtain $\sigma_{\min}(\m{H}^{\mc{I}}(\m{x}_k)) \geq \sigma_H>0$. Let $\m{r}(\m{x}_k):= \m{g}(\m{x}_k) + \m{H}(\m{x}_k)\m{d}_k$.  We have 
\begin{eqnarray}\label{cgeq:UA:b2b}
\nonumber
\|\m{d}_k\|  &=& \|  \m{Z}_{\m{A}_{\mc{I}}}  (\m{Z}_{\m{A}_{\mc{I}}}^\top  \m{H}(\m{x}_k) \m{Z}_{\m{A}_{\mc{I}}})^{-1} \big( -\m{Z}_{\m{A}_{\mc{I}}}^\top \m{g}(\m{x}_k) + \m{Z}_{\m{A}_{\mc{I}}}^\top \m{r}(\m{x}_k)\big)\| \\ 
\nonumber
& \leq &  \| \big(\m{H}^{\mc{I}}(\m{x}_k)\big)^{-1} \|\| -\m{g}^{\mc{I}}(\m{x}_k) + \m{Z}_{\m{A}_{\mc{I}}}^\top \m{r}(\m{x}_k)\|  \\
\nonumber
& \leq & \frac{1}{\sigma_H} \sqrt{ \|\m{g}^{\mc{I}}(\m{x}_k)\|^2 +\| \m{Z}_{\m{A}_{\mc{I}}}^\top\m{r}(\m{x}_k) \|^2} \\
& \leq  &\frac{\sqrt{1+ \xi_k^2}} {\sigma_H} \|\m{g}^{\mc{I}}(\m{x}_k)\|,
\end{eqnarray}
where the first inequality follows since the columns of $\m{Z}_{\m{A}_{\mc{I}}}$ form an orthonormal basis for the subspace $\mc{N}(\m{A}_{\mc{I}})$; the second inequality uses $\sigma_{\min}(\m{H}^{\mc{I}}(\m{x}_k)) \geq \sigma_H>0$; and the last inequality follows form \eqref{eqn:cg:cond}.

We combine \eqref{cgeq:super} with \eqref{cgeq:UA:b2b} to obtain
\begin{equation}\label{ccgeq:super}
\|\m{g}^{\mc{I}}(\m{x}_{k+1}) \| \leq  \xi_k  \|\m{g}^{\mc{I}}(\m{x}_k)\| + \frac{\tau_k\sqrt{1+ \xi_k^2}} {\sigma_H} \|\m{g}^{\mc{I}}(\m{x}_k)\|.
\end{equation}
In Algorithm~\ref{algo:nlco}, we have $\mc{I}=\mc{A}(\m{x}_{k})  \subset \mc{A}(\m{x}_{k+1})=\mc{J}$ which implies that
\begin{eqnarray}\label{cgeq:UA:b22}
 \|\m{g}^{\mc{J}}(\m{x}_{k+1}) \| \leq \|\m{g}^{\mc{I}}(\m{x}_{k+1})\|.
\end{eqnarray}
This inequality combined with \eqref{ccgeq:super} yields
\begin{eqnarray}\label{cgeq:quad:e}
 \|\m{g}^{\mc{A}}(\m{x}_{k+1}) \|  \leq  \xi_k  \|\m{g}^{\mc{A}}(\m{x}_k)\| + \frac{\tau_k\sqrt{1+ \xi_k^2}} {\sigma_H} \|\m{g}^{\mc{A}}(\m{x}_k)\|.
\end{eqnarray}

Now, let  $\mc{S} = \mc{A_+}(\m{x}^*)$. Since  $\m{g}^{\mc{S}}(\m{x}^*) =0$, we have
\begin{eqnarray*}
\|\m{g}^{\mc{S}}(\m{x}_{k}) \| &\leq&  \| \m{Z}_{\m{A}_{\mc{S}}}^\top (\m{H}(\m{x}^*)(\m{x}_k-\m{x}^*))\| + \|\m{Z}_{\m{A}_{\mc{S}}}^\top(\m{g}(\m{x}_{k})- \m{g}(\m{x}^*) - \m{H}(\m{x}^*) (\m{x}_k-\m{x}^*))\|.
\end{eqnarray*}
The standard estimates of the last term show that
\begin{eqnarray} \label{cgeqn:bound:denum}
\nonumber
\|\m{g}^{\mc{S}}(\m{x}_{k}) \| & \leq &  \| \m{Z}_{\m{A}_{\mc{S}}}^\top (\m{H}(\m{x}^*)(\m{x}_k-\m{x}^*))\| + \tau_k \|\m{Z}_{\m{A}_{\mc{S}}}^\top (\m{x}_k-\m{x}^*)\|  \\
\nonumber
&\leq &  \|(\m{Z}_{\m{A}_{\mc{S}}}^\top \m{H}(\m{x}^*) \m{Z}_{\m{A}_{\mc{S}}})  \m{Z}_{\m{A}_{\mc{S}}}^\top  (\m{x}_k-\m{x}^*)\| + \tau_k \| \m{Z}_{\m{A}_{\mc{S}}}^\top (\m{x}_k-\m{x}^*)\|
\\
\nonumber
&\leq &  (\kappa_H +\tau_k ) \|\m{Z}_{\m{A}_{\mc{S}}}^\top (\m{x}_k-\m{x}^*)\|\\
&\leq & (\kappa_H +\tau_k ) \|\m{x}_k-\m{x}^*\|,
\end{eqnarray}
where $$ \kappa_H = \| \m{Z}_{\m{A}_{\mc{S}}}^\top \m{H}(\m{x}^*) \m{Z}_{\m{A}_{\mc{S}}}  \| $$ and the last inequality follows since $\m{Z}_{\m{A}_{\mc{S}}}  \m{Z}_{\m{A}_{\mc{S}}}^\top  (\m{x}_k-\m{x}^*) =\m{x}_k-\m{x}^*$ for $k$ sufficiently large.

Now, from \eqref{dbar} for $k$ sufficiently large, we have
\begin{eqnarray}
\nonumber
 \|\m{x}_{k+1} - \m{x}^*\|  \leq \frac{2}{\sigma_H}  \|\m{g}^{\mc{A}}(\m{x}_{k+1}) \| ,
\end{eqnarray}
which together with \eqref{cgeq:quad:e} implies that
\begin{eqnarray*}
 \|\m{x}_{k+1} - \m{x}^*\|   \leq   \frac{2 \widehat{\kappa}_k}{\sigma_H}  \|\m{g}^{\mc{A}}(\m{x}_{k}) \|,
\end{eqnarray*}
where  
$$
 \widehat{\kappa}_k = \xi_k + \frac{\tau_k\sqrt{1+ \xi_k^2}} {\sigma_H}.  
$$
By Lemma~\ref{thm:iden}, we have  $\mc{S}=\mc{A}_+(\m{x}^*)  \subset  \mc{A}(\m{x}_k) = \mc{I}$ for all $k$ sufficiently large which yields $ \|\m{g}^{\mc{I}}(\m{x}_{k}) \| \leq  \|\m{g}^{\mc{S}}(\m{x}_{k}) \|$. Thus, 
\begin{eqnarray}\label{cgeqn:bound:num}
 \|\m{x}_{k+1} - \m{x}^*\| \leq \frac{ 2\widehat{\kappa}_k}{\sigma_H} \|\m{g}^{\mc{S}}(\m{x}_{k}) \|.
\end{eqnarray}
Combine  \eqref{cgeqn:bound:denum} with \eqref{cgeqn:bound:num} to get 
\begin{equation*}
\frac{\|\m{x}_{k+1} -\m{x}^*\|}{ \|\m{x}_k-\m{x}^*\|} \leq  M_k,
\end{equation*} 
where 
$$
M_k=\frac{ 2 \widehat{\kappa}_k}{\sigma_H}  ( \kappa_H+\tau_k ).$$
Since the sequence $\tau_k $ converges to zero, we have 
\begin{eqnarray}
 \limsup_{k \rightarrow \infty } M_k = \frac{ 2 \kappa_H \xi^*}{\sigma_H}
\end{eqnarray}
which implies that
\begin{equation*}
\limsup_{k \rightarrow \infty } \frac{\|\m{x}_{k+1} -\m{x}^*\|}{ \|\m{x}_k-\m{x}^*\|} \leq  \frac{ 2 \kappa_H \xi^*}{\sigma_H}.
\end{equation*} 
Linear convergence take places if $\xi^* \leq \frac{\sigma_H}{2\kappa_H}$, and superlinear convergence holds if $\xi^*=0.$
\end{proof}

\section{Numerical experiments}\label{sec:results}

This section compares the performance of the \NPASA{}, a MATLAB implementation using the \NLCO{} and the cyclic Barzilai--Borwein (CBB) algorithm \cite{dai2006cyclic} for the gradient projection phase, to the performance of \IPOPT{} \cite{wachter2006implementation} and \PASA{} \cite{hager2016active}.

\IPOPT{} is a software package for large-scale nonlinear constrained optimization. It implements a primal-dual interior-point algorithm with a filter line-search method for nonlinear programming and is able to use second derivatives. When the Hessian is indefinite, \IPOPT{} modifies the Hessian as $\m{H}(\m{x}) + \sigma \m{I}$ for some $\sigma>0$. The method to select $\sigma$ is fully described in \cite[Section 3.1]{wachter2006implementation}. In summary, if $\m{H}(\m{x})$ is found to be indefinite, \IPOPT{} sets $\sigma \leftarrow \sigma_0$, where $\sigma_0$ is determined from a user parameter or a previous iteration. If $\m{H}(\m{x})+ \sigma \m{I}$ is found to be indefinite, then $\sigma$ is increased by a factor $\delta$ such that $\sigma \leftarrow \delta \sigma $ and the process is repeated. Between iterations, the initial trial value $\sigma_0$ is decreased. \IPOPT{} does not explicitly compute or use directions of negative curvature. Convergence to second-order critical points is not guaranteed, but often observed in practice.

\PASA{} is a polyhedral active set algorithm for nonlinear optimization of the form \eqref{P}. \PASA{} uses the nonlinear conjugate gradient code CG\_DESCENT \cite{hager2006algorithm} for the linearly constrained optimizer and  the CBB algorithm for the gradient projection phase. The method uses the PPROJ algorithm~\cite{hager2016projection} for projecting a point onto a polyhedron. The \PASA{} algorithm does not use second derivatives and thus cannot guarantee convergence to second-order critical points. However, since it is a descent method, \PASA{} finds a minimizer most of the time.

Both \PASA{}\footnote{\url{http://users.clas.ufl.edu/hager/papers/Software/}}  and  \IPOPT{} \footnote{\url{https://github.com/coin-or/Ipopt}} codes are written in C. We use the mex (MATLAB executable) command in MATLAB to compile C code into a  MATLAB function that can be invoked in a similar fashion to prebuilt functions.  
The default parameters of \PASA{} and \IPOPT{} were used in our implementation.

\NPASA{} uses PReconditioned Iterative MultiMethod Eigensolver (\PRIMME{}) \cite{stathopoulos2010primme} to compute the leftmost eigenvalue of the matrix $\m{H}^{\mc{A}}(\m{x})$ and the corresponding eigenvector, when required. \PRIMME{} is a high-performance library for computing a few eigenvalues/eigenvectors. It is especially optimized for large and difficult problems. It can find smallest eigenvalues, and can use preconditioning to accelerate convergence. PRIMME is written in C, but complete interfaces are provided for MATLAB.  

We use MATLAB's preconditioned conjugate gradient (\texttt{pcg}) method with default tolerance $10^{-6}$ to solve the reduced systems in \NPASA{} and \NPASACG{}. \NPASA{} was stopped declaring convergence when
\begin{eqnarray*}
\|\mc{P}_\Omega (\m{x}_k - \m{g}(\m{x}_k)) -\m{x}_k \|_{\infty} \leq \epsilon_E,   
\end{eqnarray*}
where $\|\cdot \|_{\infty}$ denotes the sup-norm of a vector.  The specific parameters of \NPASA{} were chosen as
follows
\begin{eqnarray*}
\rho = 0.1, ~ \mu= 0.5,~\eta= 0.5, ~ \epsilon_{E} = 10^{-6}, ~\epsilon_{H}=10^{-4}, ~\delta_1=10^{-4}, ~ \delta_2 = 0.9.
\end{eqnarray*}

\subsection{Hamiltonian cycle problem}

The Hamiltonian cycle problem (HCP) is an important graph theory problem that features prominently in complexity theory because it is NP-complete \cite{karp1972complexity}. The definition of HCP is the following: \textit{Given a graph $\Gamma$ containing $N$ nodes, determine whether any simple cycles of length $N$ exist in the graph}.  These simple cycles of length $N$ are known as Hamiltonian cycles. If $\Gamma$ contains at least one Hamiltonian cycle (HC), we say that $\Gamma$ is a Hamiltonian graph. Otherwise, we say that $\Gamma$ is a non-Hamiltonian graph.

Next, we describe the Newton-type active set algorithm for solving HCP. To achieve this, we define decision variables $x_{ij}$ with each arc $(i,j) \in \Gamma$. Despite the double subscript, we choose to represent these variables in a decision vector $\m{x}$, with each entry corresponding to an arc in $\Gamma$. We then define $\m{P}(\m{x})$, the probability transition matrix that contains the decision variables $x_{ij}$, in matrix format where the $(i, j)$-th entry $P_{ij}({\m{x}})$ is defined as 
\begin{eqnarray}
P_{ij}({\m{x}}) & = & \left\{\begin{array}{lcl}x_{ij}, & & (i,j) \in \Gamma,\\
0, & & \mbox{otherwise.}
\end{array}
\right.\nonumber
\end{eqnarray}

The discrete nature of HCP naturally lends itself to integer programming optimization problems. However, arising from an embedding of HCP in a Markov decision process \cite{filar1994hamiltonian}, continuous optimization problems that are equivalent to HCP have been discovered in recent years. More recently,  
Ejov, Filar, Murray, and Nguyen derived an interesting continuous formulation of HCP \cite{ejov2008determinants}. In particular, it was demonstrated that if we define
\begin{eqnarray*}
\m{F}(\m{P}(\m{x})) & = & \m{I} - \m{P}(\m{x}) + \frac{1}{N}\m{e}\m{e}^\top,
\end{eqnarray*}
where $\m{e}$ is a column vector with unit entries, then HCP is equivalent to solving the following optimization problem
\begin{subequations}\label{eq:hcp}
\begin{eqnarray}\label{eq-DS0}
\min  \quad &-&\det(\m{F}(\m{P}(\m{x})))\\
\text{s.t.}\quad     &&\sum\limits_{j \in \mathcal{H}(i)}x_{ij}  = 1, \quad i = 1, \hdots, N,\label{eq-DS1}\\
\qquad \qquad \quad   && \sum\limits_{i \in \mathcal{H}(j)}x_{ij}  =  1, \quad j = 1, \hdots, N,\label{eq-DS2}\\
\qquad \qquad   \qquad  && x_{ij}  \geq  0, \quad\forall (i,j) \in \Gamma,
\label{eq-DS3}
\end{eqnarray}
\end{subequations}
where $\mathcal{H}(i)$ is the set of nodes reachable in one step from node $i$.
Constraints (\ref{eq-DS1})--(\ref{eq-DS3}) are called the {\em doubly-stochastic} constraints. We refer to constraints (\ref{eq-DS1})--(\ref{eq-DS3}) as the set $\Omega$, and we call the objective function $f(\m{P}(\m{x}))=-\det(\m{F}(\m{P}(\m{x})))$. Then, the above problem can be represented as \eqref{P}. As shown in \cite[Corollary 4.7]{ejov2008determinants}, if the graph has an HC, then it can be shown that $\m{x}^*$ is a global minimizer for \eqref{eq:hcp} if $f(\m{P}(\m{x}^*))=-N$.

Filar, Haythorpe, and Murray \cite{filar2013determinant} derived an efficient method to evaluate the first and second derivatives of $f(\m{P}(\m{x}))$. It turns out that using second derivatives and directions of negative curvature is particularly important in finding HCs \cite{filar2013determinant,henderson2012arc}. Haythorpe \cite{haythorpe2010markov} developed a specialized interior point method for \eqref{eq:hcp}, which uses second derivatives and directions of negative curvature.

In this experiment, we compared \NPASA{} with \NPASACG{}, \PASA{} and \IPOPT{} on \eqref{eq:hcp}. \NPASA{}, \NPASACG{} and \IPOPT{} solvers use second derivatives, but only \NPASA{} makes explicit use of directions of negative curvature. \PASA{} algorithm does not use second derivatives but finds a HC most of the time.

\subsubsection{Cubic graphs with 10, 12, and 14 nodes}

Cubic graphs are of interest because they represent difficult instances of HCP
\cite{michael2010}. Variables corresponding to a node with only two edges may be fixed, because the cycle is forced to use each edge. Nodes in a cubic graph all have three edges and there is no possibility for reduction. Adding more edges only increases the likelihood of an HC. Cubic graphs are the most sparse and thus are the least likely to contain HCs in general.

\begin{table}[!t]
\resizebox{.92\textwidth}{!}{
\begin{tabular}{|cccr|l|l|l|l|}
\hline
nodes &   HC & & & \NPASA & \NPASACG &  \PASA & \IPOPT  \\ \hline
10 & 17 & &HCs found &16 (94.1\%) & 15 (88.2\%)  &13 (76.4\%)&12 (70.6\%) \\
&&&  TFE &145& 158& 134 &367\\
&&& AFE &8.5 &9.2  &7.8&21.5 \\
\hline
12 & 80 & &HCs found &69 (86.2\%) & 65 (81.25\%) &63 (78.7\%) &61 (76.2\%) \\
&&& TFE & 832& 849 &917 & 2575 \\
&&& AFE &10.4& 10.6 & 11.4 & 32.18\\
\hline
14 & 474 & &HCs found &382 (80.5\%)  & 356 (75.1\%) & 330 (69.6\%) & 332 (70.4\%) \\
&&& TFE &7304 &8119& 7819 & 16615 \\
&&& AFE &15.04 &17.12& 19.2 & 35.15\\
\hline
all & 571 && HCs found &468 (81.9\%) &440 (77\%) &409 (71.2\%)& 405 (70.9\%)\\
&&& TFE &8281& 9126 &8870 &19557 \\
&&& AFE &14.5 &15.9 &15.5 & 34.25\\
\hline
\end{tabular}}
\centering
\caption{Summary of results comparing \NPASA{},  \NPASACG{},  \PASA{} and \IPOPT{} on all 10, 12, and 14 node cubic graphs with Hamiltonian cycles. The total function evaluations (TFE) reported includes the runs for which the solver failed to find a HC. The average function evaluations (AFE) are also reported over all runs. The HC column reports the total number of cubic graphs
with HCs.}
\label{tab:9cubic}
\end{table}

The results are summarized in Table \ref{tab:9cubic}. Graphs without HCs were excluded. The test set includes a total of 571 graphs. In all cases solvers were started at feasible point. From this point, \NPASA{} was able to find an HC in 81.9\% (468) of the graphs while \PASA{} and \IPOPT{} solved 71.2\% (409) and 70.9 \% (405), respectively. \NPASA{} required significantly fewer function evaluations as summarized in Table~\ref{tab:9cubic}. Furthermore, \NPASA{} with the CG algorithm (\NPASACG{}) outperforms \PASA{} and \IPOPT{} as it finds an HC in $77\%$ of the graphs. 

\subsubsection{Cubic graphs with 24, 30, and 38 nodes}

\NPASA{}, \NPASACG{}, \PASA{} and \IPOPT{} were also tested on individual cubic graphs with 24, 30, and 38 nodes as well as a 30 node graph of degree 4. The graphs contained HCs and were provided by Haythorpe \cite{michael2010}. HCP becomes more difficult as the number of nodes is increased. 

We performed our experiments using random starting points. For each graph, 20 random feasible starting points were generated in the following manner. First, a vector $\m{z}$ was produced by sampling each element from a uniform distribution over $[0,1]$. The starting points were then selected by solving
 \begin{align*}
\min_\m{x} \quad  \frac{1}{2}\| \m{z}-\m{x}\|_2^2 \quad 
 \text{s.t.} \quad \m{A}\m{x} =\m{e}, ~~\m{x} \geq 0 ,
 \end{align*}
where the linear constraints correspond to those for (\ref{eq-DS1})--(\ref{eq-DS2}). 

Starting from these randomly generated points in the feasible region, \NPASA{} found an HC in 15 (75\%) 30 node graphs of degree 4 while \IPOPT{} was able to find 8(45\%) HCs in all of these graphs when started from the same points. Similarly, \NPASA{} found an HC in 18 (90\%) 38 node graphs of degree 3. However, \IPOPT{} and \PASA{} found 4 (20\%) and 8 (40\%) HCs in these graphs, respectively. The results are summarized in Table~\ref{tab:24cubic} which shows that \NPASA{} and \NPASACG{} significantly outperform their competitors.

\begin{table}[!t]
\resizebox{.92\textwidth}{!}{
\begin{tabular}{|ccr|l|ll|ll|ll|}
\hline
nodes & degree & & \NPASA{} && \NPASACG  &&   \PASA{} & & \IPOPT{} \\ \hline
24  & 3  &  HCs found  & 14 (70\%)    && 14 (70\%) &&    12 (60\%) & & 3 (15\%)\\
&  & TFE   & 298   && 331 &&  312 &&363 \\
&  & AFE   & 14.9   &&16.5 &&  15.6 && 18.15 \\
\hline
30  & 3  &  HCs found  & 14 (70\%)   && 6 (30\%)  &&  7 (35\%) &&  3 (15\%) \\
&  & TFE  & 370   &&  411  && 455 && 1040\\
& & AFE   & 18.5   &&  20.5  &&  22.7 && 52 \\
\hline
38  & 3  &  HCs found  & 18 (90\%)   && 16 (80\%)  && 8 (40\%) && 4 (20\%)\\
&  & TFE  & 907   && 1019 && 861 && 2575\\
&  & AFE   & 45.3   && 50.95  && 43 && 128.7\\
\hline
30  & 4  &  HCs found  & 15 (75\%)   && 13 (65\%)  && 10 (50\%) && 8 (40\%) \\
&  & TFE & 546   && 671 && 712 && 1421 \\
&  & AFE    & 27.3   && 33.55  && 35.6 && 71.05\\ \hline
\end{tabular}}
\centering
\caption{Performance of \NPASA, \NPASACG{}, \PASA{}  and \IPOPT{} on specific 24, 30, and 38 node cubic graphs and a 30 node graph of degree 4. For each graph the solvers were started from 20 randomly generated, feasible starting points. The total function evaluations (TFE) and average function evaluations (AFE) were reported.}
\label{tab:24cubic}
\end{table}

\subsection{The CUTEst problems}

In this section, some experiments are presented for a MATLAB implementation applied to a set of problems from CUTEst \cite{bongartz1995cute,gould2003cuter}. 

This experiment compares algorithms on 55 problems with a nonlinear objective, polyhedral constraints and up to 7564 variables. The number of function evaluations was used for evaluation. 
In running the numerical experiments, we checked whether different codes converged to different local minimizers; when comparing the codes, we restricted ourselves to test problems in which all codes converged to the same local minimizer.

The performance of the algorithms was evaluated using the performance profiles of Dolan and Mor\'{e} \cite{dolan2002benchmarking}. Denote the set of solvers by $\mc{S}$ and the set of test problems by $\mc{P}$. The metric $t_{p,s}$ indicates the performance of solver $s \in \mc{S}$ on problem $p \in \mc{P}$. The metric could be running time, number of function evaluations, or solution quality. The best performance on any problem is given by
$t_{p,\min}= \min \{ t_{p,s}:s\in  \mc{S}\}$. The ratio $r_{p,s}= t_{p,s}/ t_{p,\min}$ indicates the relative performance of solver $s$ on $p$. We see $r_{p,s}=1$ if $s$ achieved the best observed performance on $p$ and $r_{p,s}>1$ otherwise. If solver $s$ failed on $p$, then
$r_{p,s}$ can be set to a sufficiently large number. Finally, performance profiles are plots of the function
\begin{align*}
\rho_s (\tau)= \dfrac{|\{p:r_{p,s} \leq  \tau \}|}{|\mc{P}|}
\end{align*}
which is the fraction of problems that $s$ was able to solve with performance ratio at most
$\tau$. The profile functions $\rho_s(\tau) $ are drawn for all
$s \in  \mc{S}$ on the same plot. Solvers with greater area under the
profile curve exhibit better relative performance on the test set.
\begin{figure}[t]
	\centering \makebox[0in]{
\begin{tabular}{c}
\includegraphics[scale=0.18]{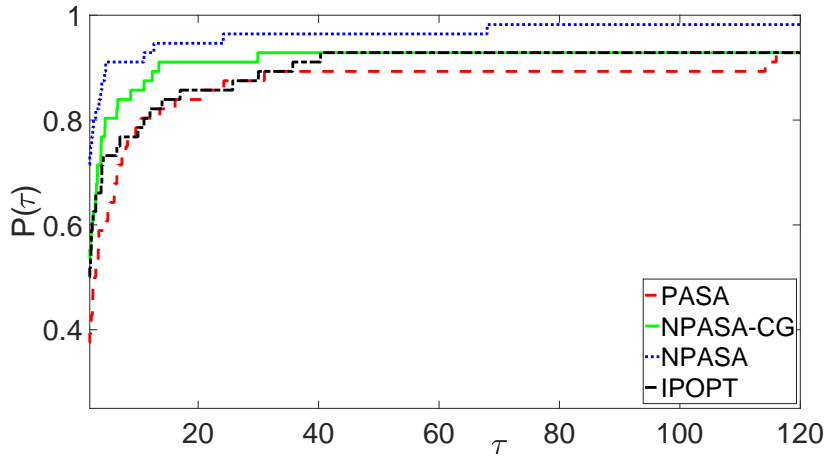}
\end{tabular}
}
\caption{Performance profiles, function evaluation metric, 55 CUTEst problems.}\label{fig:polyfunc}
\end{figure}

Figure~\ref{fig:polyfunc} gives the performance profiles based on the total number of function evaluations required to solve the 55 polyhedral constrained problems. \NPASA{} has the best performance in this metric among the algorithms, as it solves about 70\% of test problems with the least number of function evaluations. Note that \NPASACG{} has the second best performance. These results support the conclusion that \NPASA{} or \NPASA{} with the CG algorithm is preferable to \PASA{} if the Hessian matrix can be obtained explicitly. 
%

\section{Conclusions}\label{sec:conclusions}

In this paper, we introduced a Newton-type active set algorithm for large-scale minimization subject to polyhedral constraints. The algorithm consists of a gradient-projection step, a modified Newton step in the null space of the constraint matrix, and a set of rules for branching between the two steps. We showed that the proposed method asymptotically takes the modified Newton step when the active constraints are linearly independent and a strong second-order sufficient optimality condition hold. We also showed that the method has a quadratic rate of convergence under standard conditions. 
%
\section*{Acknowledgments}
The authors are very grateful to Nicholas Wayne Henderson and Walter Murray for making their code of HCP available to us. 
\section*{Appendix}
\subsection*{Proof of Lemma~\ref{lem:sk}}
\begin{proof}
By our assumption $\phi^{'}(0) < 0$, or $\phi^{'}(0) \leq 0$  and $\phi^{''}(0) < 0$. This, together with \eqref{eq:linf} implies that
\begin{equation}\label{eq:psineg}
\psi(s)s = \phi^{'}(0)s+ \frac{1}{2} \min \{ \phi^{''}(0) , 0\} s^2 \leq 0, \quad \text{for all}  \quad s \in S.
\end{equation}
Now, let 
 \begin{equation*}
t = \sup \{s \in S  : \phi (s) \leq \phi(0)\} .
\end{equation*}
From \eqref{eq:psineg} and \eqref{eq:line1}, we get $t > 0$. Further, it follows from  the compactness assumption and the continuity of the function $\phi$ that $ t$ is finite and $\phi(0) = \phi(t)$. Now, from \eqref{eq:linf}, we obtain 
\begin{align}\label{eq:line}
\nonumber
\phi(t) & \geq  \phi(0)+   \delta_1 \big(\phi^{'}(0) t+ \frac{1}{2} \min \{ \phi^{''}(0) , 0\} t^2\big)\\
        & =  \phi(0) + \delta_1 \psi(t)t.
\end{align}
Define $h : S \rightarrow \mathbb{R}$ by  
\begin{equation*}
h(s) =  \phi(s) -\phi(0) - \delta_2 \psi(s)s,
\end{equation*}
where $\psi(s) = \phi^{'}(0) + \frac{1}{2} \min \{ \phi^{''}(0) , 0\} s$.

Since $\delta_1 \leq \delta_2$, we have $h(s) \geq  0$. Note also that $h(0)=0$ and either $h^{'}(0)<0 $, or $h^{'}(0) \leq 0$ and $h^{''}(0) <0$. This, together with the continuity of $h$ implies the existence of $\bar{t} \in (0, t]$ such that 
\begin{equation*}
h(\bar{t}) = h(0)= 0, \quad  \text{and} \quad h(s) < 0  \quad \text{for all}  \quad s \in (0, \bar{t}).
\end{equation*}

Now, it follows from Rolle's theorem that there exists at least one $s$ in the open interval $ (0, \bar{t})$ with $h^{'} (s)=0$, and thus \eqref{eq:line2} holds. Further, $h(s) < 0$ and $\delta_1 \leq \delta_2$ imply \eqref{eq:line1}.
\end{proof}
\bibliographystyle{siamplain}
\bibliography{siam_ref}

\end{document}